\documentclass{article}
\usepackage{euscript,amsmath,amssymb,amsfonts,graphicx,bm,amsthm,mathtools}

  \usepackage{paralist}
  \usepackage{graphics} 
\usepackage[caption=false]{subfig}
\usepackage{soul}

\usepackage{latexsym,amssymb,enumerate,amsmath,amsthm} 
  \usepackage{graphicx} 
  \usepackage{tikz}
     \usepackage[colorlinks=false]{hyperref}
 \hypersetup{urlcolor=blue, citecolor=red}
\usepackage{latexsym,amssymb,enumerate,amsmath} 
\usepackage{pgfplots}
\usepackage{soul,xcolor}

\newcommand{\dd}{\textup{d}}
\def\eps{\varepsilon}
\def\E{\mathbb{E}}
\def\P{\mathbb{P}}
\def\R{\mathbb{R}}

\def\PP{\mathcal{P}}
\def\dmin{d_{\textup{min}}}
\def\amin{t_{\textup{min}}}
\def\tmin{t_{\textup{min}}}

\def\t{\textup{target}}
\def\diff{\textup{diff}}
\def\ctmc{\textup{ctmc}}
\def\sub{\textup{sub}}
\def\ctrw{\textup{ctrw}}

\def\D{\prescript{}{0}D_{t}^{1-\alpha}}

\def\L{\mathcal{L}}

\def\dist{\textup{d}}

\newtheorem{theorem}{Theorem}
\newtheorem{lemma}[theorem]{Lemma}
\newtheorem{proposition}[theorem]{Proposition}
\newtheorem{corollary}[theorem]{Corollary}
\theoremstyle{plain}
\theoremstyle{remark}

\theoremstyle{definition}
\newtheorem{definition}[theorem]{Definition}

\begin{document}


\title{Extreme first passage times for random walks on networks}


\author{Sean D. Lawley\thanks{Department of Mathematics, University of Utah, Salt Lake City, UT 84112 USA (\texttt{lawley@math.utah.edu}).}
}
\date{\today}
\maketitle

\begin{abstract}
Many biological, social, and communication systems can be modeled by ``searchers'' moving through a complex network. For example, intracellular cargo is transported on tubular networks, news and rumors spread through online social networks, and the rapid global spread of infectious diseases occurs through passengers traveling on the airport network. To understand the timescale of search/transport/spread, one commonly studies the first passage time (FPT) of a single searcher (or ``transporter'' or ``spreader'') to a target. However, in many scenarios the relevant timescale is not the FPT of a single searcher to a target, but rather the FPT of the fastest searcher to a target out of many searchers. For example, many processes in cell biology are triggered by the first molecule to find a target out of many, and the time it takes an infectious disease to reach a particular city depends on the first infected traveler to arrive out of potentially many infected travelers. Such fastest FPTs are called extreme FPTs. In this paper, we study extreme FPTs for a general class of continuous-time random walks on networks (which includes continuous-time Markov chains). In the limit of many searchers, we find explicit formulas for the probability distribution and all the moments of the $k$th fastest FPT for any fixed $k\ge1$. These rigorous formulas depend only on network parameters along a certain geodesic path(s) from the starting location to the target since the fastest searchers take a direct route to the target. Hence, the extreme FPTs are independent of the details of the network outside this geodesic(s) and can be drastically faster and less variable than conventional FPTs of single searchers. Furthermore, our results allow one to estimate if a particular system is in a regime characterized by fast extreme FPTs. We also prove similar results for mortal searchers on a network that are conditioned to find the target before a fast inactivation time. We illustrate our results with numerical simulations and uncover potential pitfalls of modeling diffusive or subdiffusive processes involving extreme statistics.
\end{abstract}

\section{\label{intro}Introduction}

Networks are used to model many biological, social, and communication systems \cite{albert2002, newman2003, pastor2015}. A network (or graph) consists of (i) nodes (or vertices) which can represent cells, individuals, cities, computers, etc., and (ii) edges between nodes which represent their interactions. An important area of network science studies \emph{random walks on networks}, which are stochastic processes involving ``searchers'' (or ``transporters'' or ``spreaders'') who explore a network by randomly moving between connected nodes \cite{noh2004, masuda2017}. Random walks on networks have been used to model very diverse dynamical systems, including the spread of infectious diseases \cite{iannelli2017}, intracellular transport \cite{dora2020}, animal foraging, the spread of opinions and rumors, and node ranking (i.e.\, PageRank) \cite{masuda2017}.

First passage times (FPTs) are commonly used to understand the timescale of search/transport/spread in such models \cite{noh2004, masuda2017, redner2001}. A FPT is defined as the first time a searcher reaches a given target(s). Mathematically, if $X(t)$ denotes the position of a searcher in a discrete set of nodes $I$ at time $t\ge0$, then the FPT to some set of target nodes $I_{\t}\subset I$ is defined by
\begin{align}\label{tau0}
\tau
:=\inf\{t>0:X(t)\in I_{\t}\}.
\end{align}
Important advances have been made to compute the statistics of such a FPT, and in particular to understand how the FPT depends on the structure of the network \cite{noh2004, condamin2007b, reuveni2010, masuda2017}.

However, in many applications the important timescale is not the FPT of a given single searcher to a target. Instead, the relevant timescale is the FPT of the fastest searcher to find a target out of many searchers. For example, in the context of epidemics spreading globally between cities, the first time the disease reaches a particular locale depends on the first infected traveler to arrive out of potentially many infected travelers. As another example, it was recently pointed out that the transport efficiency of the endoplasmic reticulum network depends on the time it takes the fastest molecule to travel between nodes out of many molecules \cite{dora2020}. In fact, many events in cell biology are triggered by the fastest searcher out of many searchers (see the review \cite{schuss2019} and subsequent commentaries \cite{coombs2019, redner2019, sokolov2019, rusakov2019, martyushev2019, tamm2019, basnayake2019c}).

Such fastest FPTs are an example of extreme statistics \cite{colesbook, falkbook, haanbook} and are thus called extreme FPTs \cite{lawley2020uni}. Mathematically, an extreme FPT is defined as
\begin{align*}
T_{N}
:=\min\{\tau_{1},\dots,\tau_{N}\},
\end{align*}
where $\tau_{1},\dots,\tau_{N}$ are $N$ FPTs as in \eqref{tau0}. In particular, $\tau_{1},\dots,\tau_{N}$ represent the FPTs of $N$ searchers who are searching simultaneously for a target, and thus $T_{N}$ is the FPT of the fastest searcher. Typically, the $N$ searchers are assumed to be homogeneous and noninteracting, in which case $\tau_{1},\dots,\tau_{N}$ are independent and identically distributed (iid). More generally, some systems depend on the $k$th fastest FPT for $k\in\{1,\dots,N\}$,
\begin{align*}
T_{k,N}
:=\min\big\{\{\tau_{1},\dots,\tau_{N}\}\backslash\cup_{j=1}^{k-1}\{T_{j,N}\}\big\},
\end{align*}
where $T_{1,N}:=T_{N}$. For searchers which move by diffusion through a continuous state space, many works have studied extreme FPTs \cite{weiss1983, yuste2000, yuste2001, redner2014, meerson2015, ro2017, godec2016x, hartich2018, hartich2019, basnayake2019, lawley2020esp, lawley2020uni, lawley2020dist, lawley2020sub}. However, much less is known about extreme FPTs for processes on discrete state spaces. 

In this paper, we study extreme FPTs for $N\gg1$ searchers on a network. We find explicit and rigorous approximations of the full probability distribution and all the moments of $T_{N}$ for large $N$. We obtain these results by proving that a certain rescaling of $T_{N}$ converges in distribution to a Weibull random variable. In addition, we generalize these results to the $k$th fastest FPT, $T_{k,N}$, for any $k\ll N$.

In our model, each searcher moves according to a continuous-time random walk (CTRW). To describe how such a CTRW searcher moves through a network, one must specify (i) where a searcher moves and (ii) when a searcher moves. For (i), we assume that the discrete-time process obtained by observing the sequence of nodes visited by the searcher is a discrete-time Markov chain. For (ii), we first assume that the waiting time between moves is exponentially distributed. In this case, the searcher is a continuous-time Markov chain (CTMC). We then generalize (ii) by allowing the waiting time to come from a general family of probability distributions. 

\begin{figure}[t]
\centering
\includegraphics[width=.6\linewidth]{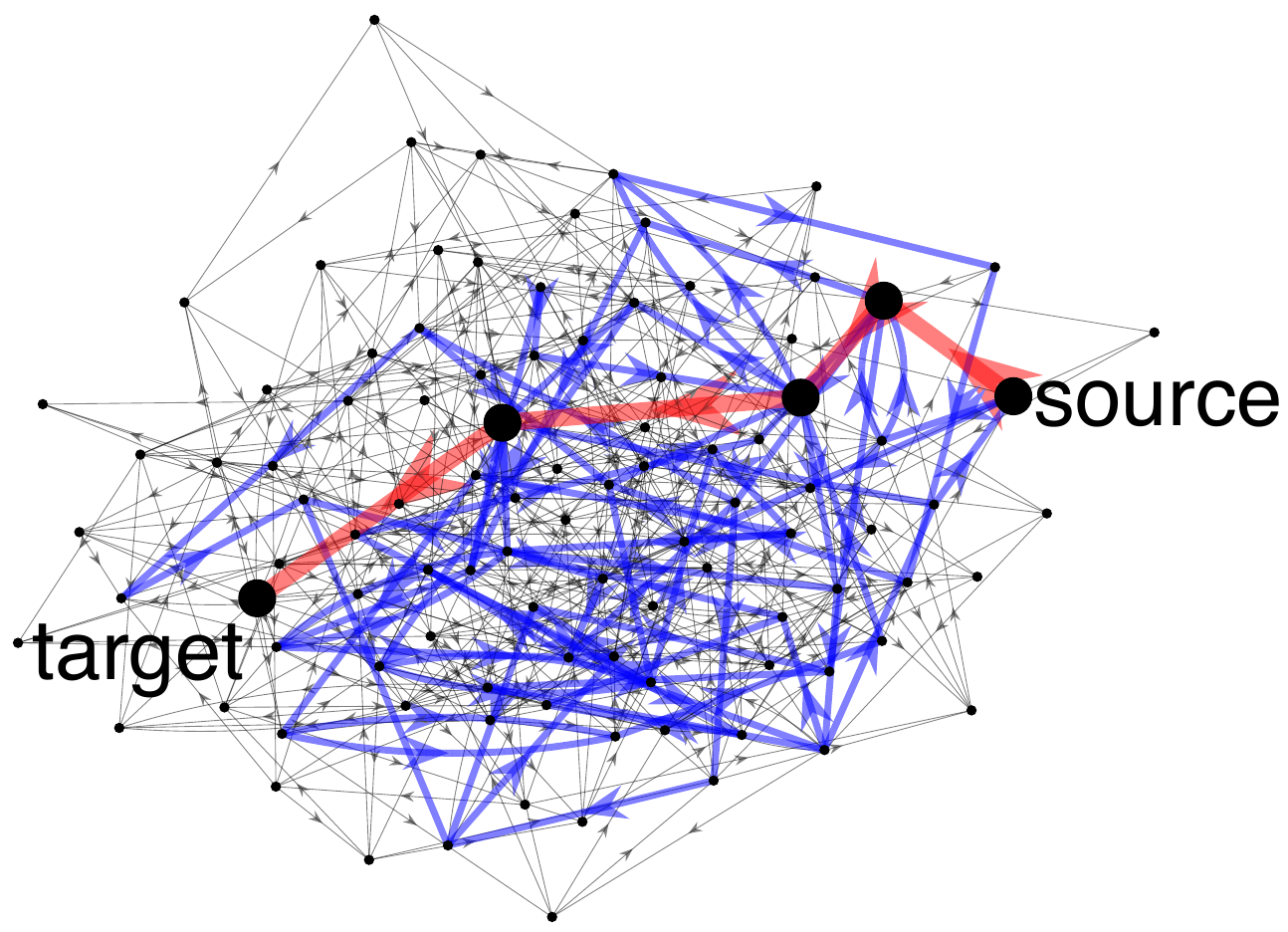}
\caption{\small While a typical searcher/spreader (blue trajectory) wanders around the network before reaching the target, the fastest searcher/spreader (red trajectory) out of many follows a certain geodesic path(s) from its initial position (source) to the target.}
\label{figschem}
\end{figure}

Our results show that extreme FPTs are much faster than single FPTs,
\begin{align}\label{fast}
T_{N}\ll \tau,
\end{align}
less variable than single FPTs, and independent of much of the details of the network. In particular, while a typical searcher explores much of the network before finding the target, the fastest searcher follows a certain geodesic path(s) from its starting location to the target. See Figure~\ref{figschem} for an illustration. Indeed, our explicit formulas for the distribution and moments of extreme FPTs depend only on the parameters along this geodesic path(s). Furthermore, our moment formulas are accompanied by rigorous convergence rates, which thus allow one to estimate when a particular system is in the extreme regime in \eqref{fast}. That is, while it is quite intuitive that $T_{N}\ll\tau$ for sufficiently large $N$, our analysis determines quantitatively what constitutes ``large $N$,'' and estimates $T_{N}$ in this large $N$ regime.

In addition, we prove qualitatively similar results for so-called \emph{mortal} searchers \cite{abad2010, abad2012, abad2013, yuste2013, meerson2015b, meerson2015, grebenkov2017, ma2020, lawley2020mortal}. Mortal searchers cannot search for the target indefinitely, but rather may be inactivated (degrade/die/evanesce/etc.)\ before finding the target. We find formulas for the moments of the FPT of a single searcher who is conditioned to find the target before a fast inactivation time. Similar to our results for extreme FPTs and the results in \cite{ma2020, lawley2020mortal} for diffusive mortal searchers, we find that such conditional searchers take a direct route the target. 

The rest of the paper is organized as follows. In section~\ref{ctmc}, we analyze the case that the searchers are CTMCs, which means the waiting times are exponentially distributed. In section~\ref{general}, we extend our analysis to more general waiting time distributions. In section~\ref{numerics}, we compare our results to numerical simulations on complex networks. In section~\ref{diffusion}, we use our results to investigate the well-studied problem of extreme FPTs for diffusive searchers. We uncover some potential pitfalls of modeling diffusive or subdiffusive processes involving extreme statistics. In section~\ref{mortal}, we consider mortal searchers. We conclude by discussing our results in the context of several related works. 

\section{\label{ctmc}Exponential waiting times}

\subsection{Random walk setup}

Let $X=\{X(t)\}_{t\ge0}$ be a CTMC on a finite or countably infinite state space $I$. The process $X$ is a single searcher (random walker), and the state space $I$ is the nodes (vertices) of the network (graph). There is a directed edge from $i\in I$ to $j\in I$ if $X$ can jump directly from $i$ to $j$. We refer to the time it takes $X$ to jump directly from $i$ to $j$ as the waiting time. Since $X$ is a CTMC, such waiting times are exponentially distributed. 

The dynamics of $X$ are described by its infinitesimal generator matrix \cite{norris1998},
\begin{align*}
Q=\{q(i,j)\}_{i,j\in I}.
\end{align*}
The off-diagonal entries of $Q$ are nonnegative,
\begin{align*}
q(i,j)
\ge0,\quad i\neq j,
\end{align*}
and give the rate that $X$ jumps from state $i\in I$ to state $j\in I$. The diagonal entries of $Q$ are nonpositive,
\begin{align*}
q(i,i)\le0,\quad i\in I,
\end{align*}
and are chosen so that $Q$ has zero row sums,
\begin{align}\label{od2}
\sum_{j\in I}q(i,j)=0,\quad i\in I.
\end{align}
It is convenient to define
\begin{align*}
q(i)
:=-q(i,i)\ge0,\quad i\in I,
\end{align*}
which is the total rate that $X$ leaves state $i\in I$ (regardless of the state that $X$ jumps to). We assume that
\begin{align*}
\sup_{i\in I}q(i)<\infty,
\end{align*}
which ensures that $X$ cannot take infinitely many jumps in finite time.

\subsection{Single FPTs}

Define the FPT of $X$ to some target set $I_{{\t}}\subset I$,
\begin{align}\label{tau}
\tau
:=\inf\{t>0:X(t)\in I_{{\t}}\}.
\end{align}
Denote the initial distribution of $X$ by
\begin{align*}
\rho
=\{\rho(i)\}_{i\in I}
=\{\P(X(0)=i)\}_{i\in I}.
\end{align*}
To avoid trivial cases, we assume that $X$ cannot start directly on the target, which means
\begin{align}\label{away}
I_{{\t}}
\cap\textup{supp}(\rho)
=\varnothing,
\end{align}
where $\textup{supp}(\rho)\subset I$ denotes the support of the distribution $\rho$,
\begin{align*}
\textup{supp}(\rho)
:=\{i\in I:\rho(i)>0\}.
\end{align*}

We refer to $\tau$ as a single FPT, since it is the FPT of a given single searcher $X$. We are interested in studying the fastest (extreme) FPT out of $N\gg1$ searchers,
\begin{align*}
T_{N}
:=\min\{\tau_{1},\dots,\tau_{N}\},
\end{align*}
where $\tau_{1},\dots,\tau_{N}$ are $N$ iid realizations of $\tau$. Since $\tau_{1},\dots,\tau_{N}$ are iid, it is immediate that the distribution of $T_{N}$ is
\begin{align}\label{obv}
\P(T_{N}>t)
=\big(\P(\tau>t)\big)^{N}.
\end{align}

Since $\P(\tau>t)$ decreases monotonically in $t>0$, it is intuitively clear from \eqref{obv} that the large $N$ distribution of $T_{N}$ is determined by the short-time distribution of $\tau$. In this subsection, we find this short-time distribution.
In particular, we prove in Proposition~\ref{cdf} below that 
\begin{align}\label{basic}
\P(\tau\le t)
\sim\frac{\Lambda}{d!}t^{d}\quad\text{as }t\to0+,
\end{align}
where (i) $d\ge1$ is the smallest number of jumps that $X$ must take to reach $I_{{\t}}$ and (ii) $\Lambda>0$ is a sum of the  products of the jump rates along the shortest paths from $\textup{supp}(\rho)$ to $I_{\t}$ (where the terms in the sum are weighted according to $\rho$). In the remainder of this subsection, we make \eqref{basic} precise.

Define a path $\PP$ of length $d\in\mathbb{Z}_{\ge0}$ from a state $i_{0}\in I$ to a state $i_{d}\in I$ to be a sequence of $d+1$ states in $I$,
\begin{align}\label{path}
\PP
=(\PP(0),\dots,\PP(d))
=(i_{0},i_{1},\dots,i_{d})\in I^{d+1},
\end{align}
so that
\begin{align}\label{explain}
q(\PP(k),\PP(k+1))
>0,\quad\text{for }k\in\{0,1,\dots,d-1\}.
\end{align}
In words, \eqref{explain} means that there is a strictly positive probability that $X$ may traverse the path $\PP$. Naturally, we assume that there is a path from the support of $\rho$ to the target,
\begin{align}\label{canfind}
\bigcup_{d\ge1}\{\PP\in I^{d+1}:\PP(0)\in\text{supp}(\rho),\PP(d)\in I_{\t}\}\neq\varnothing.
\end{align}
If \eqref{canfind} is violated, then $\tau=T_{N}=\infty$ almost surely and the problem is trivial.

For a path $\PP\in I^{d+1}$, define $\lambda(\PP)$ to be the product of the rates along the path,
\begin{align}\label{lambda}
\lambda(\PP)
:=\prod_{i=0}^{d-1}q(\PP(i),\PP(i+1))>0.
\end{align}
Let $\dmin(I_{0},I_{1})\in\mathbb{Z}_{\ge0}$ denote the length of the geodesic path from a set of nodes $I_{0}\subset I$ to another set of nodes $I_{1}\subset I$,
\begin{align}\label{d}
\dmin(I_{0},I_{1})
:=
\inf\{d:\PP\in I^{d+1},\PP(0)\in I_{0},\PP(d)\in I_{1}\}. 
\end{align}
In words, $\dmin(I_{0},I_{1})$ is the smallest number of jumps required for $X$ to move from $I_{0}$ to $I_{1}$.

Define the set of all paths from $I_{0}$ to $I_{1}$ with the minimum length $\dmin(I_{0},I_{1})$ in \eqref{d},
\begin{align}\label{S}
\begin{split}
\mathcal{S}(I_{0},I_{1})
:=\{\PP\in I^{d+1}:\PP(0)\in I_{0},\PP(d)\in I_{1},d=\dmin(I_{0},I_{1})\}.
\end{split}
\end{align}
Define
\begin{align}\label{Lambdathm}
\Lambda(\rho,I_{1})
:=\sum_{\PP\in\mathcal{S}(\textup{supp}(\rho),I_{1})}\rho(\PP(0))\lambda(\PP).
\end{align}
The quantity $\Lambda(\rho,I_{1})$ is easiest to understand by first considering the case that $\rho(i_{0})=1$ for some $i_{0}\in I$ (meaning $\text{supp}(\rho)=i_{0}=X(0)$ almost surely). In this case, if there is a unique path with the minimum number of jumps $\dmin(i_{0},I_{1})$, then $\Lambda(\rho,I_{1})$ is simply the product of the jump rates along this geodesic path ($\lambda(\PP)$ in \eqref{lambda}). If there are multiple such geodesic paths, then $\Lambda(\rho,I_{1})$ sums the products of the jump rates along these paths. Finally, if the support of $\rho$ is not concentrated at a single point, then $\Lambda(\rho,I_{1})$ simply sums the products of the jump rates along all the geodesic paths, where the sum is weighted according to the initial distribution $\rho$.

With these definitions in place, we can now give the short-time behavior of the distribution of $\tau$. Throughout this paper,
\begin{align*}
\text{``$f\sim g$'' means $f/g\to1$.}
\end{align*}

\begin{proposition}\label{cdf}
We have that
\begin{align*}
\P(\tau\le t)
\sim\frac{\Lambda}{d!}t^{d}\quad\text{as }t\to0+,
\end{align*}
where 
\begin{align*}
d
&=\dmin(\textup{supp}(\rho),I_{{\t}})\in\mathbb{Z}_{>0},\\
\Lambda
&=\Lambda(\rho,I_{{\t}})>0.
\end{align*}
\end{proposition}

The proof of Proposition~\ref{cdf}, as well as the proofs of the theorems and propositions below, are given in the appendix.

\subsection{Fastest FPT}

Having determined the short-time distribution of a single FPT $\tau$ in Proposition~\ref{cdf}, we now determine the large $N$ distribution and moments of the fastest FPT $T_{N}$. In Theorem~\ref{main} below, we prove that a certain rescaling of $T_{N}$ converges in distribution to a Weibull random variable. Before stating the theorem, we recall two requisite definitions. 

\begin{definition}
A sequence of random variables $\{Z_{N}\}_{N\ge1}$ \emph{converges in distribution} to a random variable $Z$ if
\begin{align*}
\P(Z_{N}\le z)\to\P(Z\le z)\quad\text{as }N\to\infty,
\end{align*}
for all points $z\in\R$ such that $F(z):=\P(Z\le z)$ is continuous. If this holds, then we write
\begin{align*}
Z_{N}\to_{\dist}Z\quad\text{as }N\to\infty.
\end{align*}
\end{definition}

\begin{definition}
A random variable $Z\ge0$ has a \emph{Weibull distribution} with scale parameter $t>0$ and shape parameter $d>0$ if
\begin{align}\label{zweibull}
\P(Z>z)
=\exp(-(z/t)^{d}),\quad z\ge0.
\end{align}
If \eqref{zweibull} holds, then we write
\begin{align*}
Z=_{\textup{d}}\textup{Weibull}(t,d).
\end{align*}
\end{definition}

\begin{theorem}\label{main}
Let $d\ge1$ and $\Lambda>0$ be as in Proposition~\ref{cdf} and define
\begin{align*}
A
&=\frac{\Lambda}{d!}>0.
\end{align*}
The following rescaling of $T_{N}$ converges in distribution to a Weibull random variable,
\begin{align}\label{cd}
(AN)^{1/d}T_{N}
\to_{\dist}
\textup{Weibull}(1,d)\quad\text{as }N\to\infty.
\end{align}
Suppose further that
\begin{align*}
\E[T_{N}]<\infty\quad\text{for some }N\ge1.
\end{align*}
Then for each moment $m\in(0,\infty)$, we have that
\begin{align}\label{momentformula}
\E[(T_{N})^{m}]
&\sim \frac{\Gamma(1+m/d)}{(AN)^{m/d}}\quad\text{as }N\to\infty.
\end{align}
\end{theorem}

The convergence in \eqref{cd} means roughly that the distribution of $T_{N}$ for large $N$ is given by
\begin{align}\label{rough}
T_{N}\approx_{\dist}\textup{Weibull}\big((AN)^{-1/d},d\big). 
\end{align}
Further, the general formula for the $m$th moment in \eqref{momentformula} means that as $N\to\infty$ the mean and variance are
\begin{align}\label{mv}
\begin{split}
\E[T_{N}]
&\sim \frac{\Gamma(1+\frac{1}{d})}{(AN)^{1/d}},\\
\textup{Variance}(T_{N})
&\sim\frac{\Gamma(1+\frac{2}{d})
-\big(\Gamma(1+\frac{1}{d})\big)^{2}}{(AN)^{2/d}}.
\end{split}
\end{align}

Compared to a single FPT, Theorem~\ref{main} implies that extreme FPTs are (i) faster, (ii) less variable, and (iii) less affected by the size/structure/details of the network. To see points (i) and (ii), note the vanishing mean and variance in \eqref{mv} (which is implied by the vanishing moments in \eqref{momentformula}). To see point (iii), notice that Theorem~\ref{main} implies that the limiting distribution of $T_{N}$ is completely determined by the parameters $N$, $\Lambda$, and $d$. In particular, the only network parameters that enter into the large $N$ distribution of $T_{N}$ are along the geodesic path(s) from the initial distribution to the target. Therefore, if there are many searchers ($N\gg1$), the distribution of $T_{N}$ is unaffected by changes to the network outside this geodesic path(s). This is illustrated in Figure~\ref{figsame}, which depicts two vastly different networks that nonetheless have the same extreme FPT distributions.

\begin{figure}[t]
\centering
\includegraphics[width=.6\linewidth]{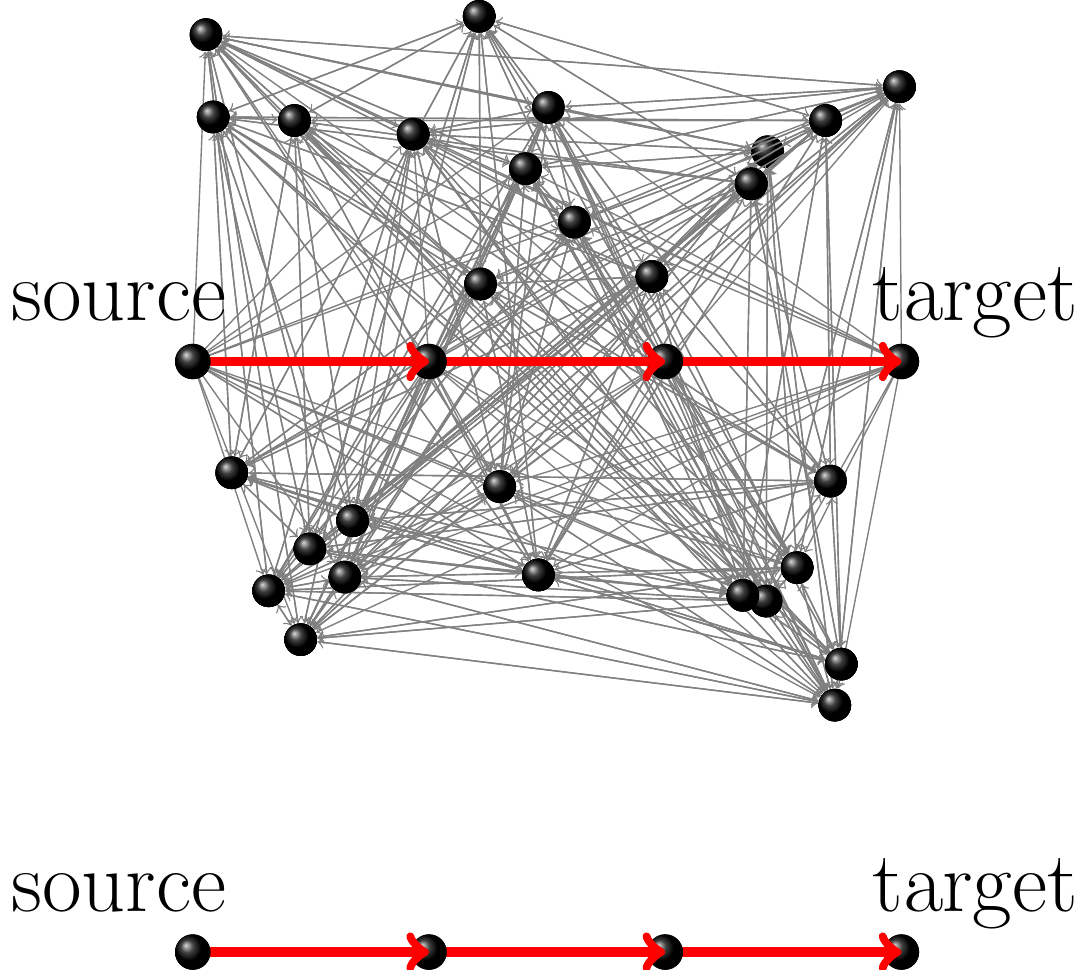}
\caption{\small The fastest searcher out of $N\gg1$ searchers moving through the top network takes the geodesic path from initial position (source) to the target, which is highlighted in red. Theorems~\ref{main} and \ref{kth} imply that the statistics of the extreme FPTs depend only on network parameters along this geodesic path. Therefore, deleting the nodes outside this geodesic path does not affect the extreme FPTs. Hence, the complicated top network and the simple bottom network have the same extreme FPTs for large $N$ (assuming the rates along the geodesic path are identical in the top and bottom networks).
}
\label{figsame}
\end{figure}

In addition, \eqref{momentformula} allows us to estimate when $T_{N}$ is in the extreme Weibull regime of Theorem~\ref{main}. In particular, we are assured by \eqref{momentformula} that $T_{N}$ is in this large $N$ regime for all moments $m\ge1$ if $(AN)^{-1/d}$ is much smaller than the other timescales in the problem, which means
\begin{align*}
N
\gg\frac{(\sup_{i\in I}q(i))^{d}}{A}
=\frac{d!(\sup_{i\in I}q(i))^{d}}{\Lambda},
\end{align*}
since $\sup_{i\in I}q(i)<\infty$ is the fastest rate in the problem.

\subsection{$k$th fastest FPT}

We now generalize Theorem~\ref{main} on the fastest FPT to the $k$th fastest FPT,
\begin{align*}
T_{k,N}
:=\min\big\{\{\tau_{1},\dots,\tau_{N}\}\backslash\cup_{j=1}^{k-1}\{T_{j,N}\}\big\},
\end{align*}
where $T_{1,N}:=T_{N}$. The large $N$ distribution of $T_{k,N}$ is described in terms of a generalized Gamma random variable.

\begin{definition}
A random variable $X\ge0$ has a \emph{generalized Gamma distribution} with parameters $t>0$, $d>0$, $k>0$ if
\begin{align}\label{gGamma}
\P(Z>z)
=\frac{\Gamma(k,(z/t)^{d})}{\Gamma(k)},\quad z\ge0,
\end{align}
where $\Gamma(a,z):=\int_{z}^{\infty}u^{a-1}e^{-u}\,\dd u$ denotes the upper incomplete gamma function. If \eqref{gGamma} holds, then we write
\begin{align*}
Z=
_{\textup{d}}\textup{gen}\Gamma(t,d,k).
\end{align*}
\end{definition}

\begin{theorem}\label{kth}
Fix $k\ge1$ and let $A$ be as in Theorem~\ref{main}. The following rescaling of $T_{k,N}$ converges in distribution to a generalized Gamma random variable,
\begin{align*}
(AN)^{1/d}T_{k,N}
\to_{\dist}
\textup{gen}\Gamma(1,d,k)\quad\text{as }N\to\infty.
\end{align*}
Suppose further that
\begin{align*}
\E[T_{N}]<\infty\quad\text{for some }N\ge1.
\end{align*}
Then for each moment $m\in(0,\infty)$, we have that
\begin{align}\label{momentformula2}
\E[(T_{k,N})^{m}]
&\sim\frac{\Gamma(k+m/d)/\Gamma(k)}{(AN)^{m/d}}
\quad\text{as }N\to\infty.
\end{align}
\end{theorem}

Similar to \eqref{rough}, Theorem~\ref{kth} means roughly that the distribution of $T_{k,N}$ for large $N$ is
\begin{align*}
T_{k,N}
\approx_{\dist}
\textup{gen}\Gamma\big((AN)^{-1/d},d,k\big). 
\end{align*}
Further, the general formula for the $m$th moment in \eqref{momentformula2} means that as $N\to\infty$ the mean and variance are
\begin{align*}
\E[T_{k,N}]
&\sim\frac{\Gamma(k+\frac{1}{d})/\Gamma(k)}{(AN)^{1/d}},\\
\textup{Variance}(T_{k,N})
&\sim\frac{\Gamma(k+\frac{2}{d})/\Gamma(k)
-\big(\Gamma(k+\frac{1}{d})/\Gamma(k)\big)^{2}}{(AN)^{2/d}}.
\end{align*}

\section{\label{general}General waiting times}

\subsection{Random walk setup}

In the previous section, we took the random walkers to be CTMCs. From a modeling perspective, there are three key restrictions implied by this assumption. First, the waiting times are always exponentially distributed. Second, the waiting times depend only on the current state. That is, 
the waiting time to jump from node $i$ to node $j$ depends only on $i$ and not on the destination $j$ (the waiting time is exponential with rate $q(i)$ even though the ``jump rate'' is $q(i,j)$, see below). Naturally, in many applications the time it takes to move depends on the destination. Third, the waiting times may be arbitrarily fast. That is, for any small time $\eps>0$, there is a strictly positive probability that the searcher will jump from $i$ to $j$ in a time less than $\eps$ (as long as $q(i,j)>0$). In this section, we remove these three restrictions. 

Let $X=\{X(t)\}_{t\ge0}$ be a continuous-time stochastic process on a discrete state space $I$ (as a technical point, we assume paths of $X$ are continuous from the right). In contrast to section~\ref{ctmc}, we assume for simplicity that $I$ is finite. Informally, we suppose that the CTRW $X$ walks on the network $I$ in the following manner. From a state $i\in I$, the walker chooses its next state according to a probability distribution that depends on only its current state $i$. Then, having chosen that the next state is some $j\in I$, the walker waits at its current state until time $S>0$, where $S$ is chosen according to a probability distribution that may depend on both the current state $i$ and the next state $j$. 

More precisely, let $J(0),J(1),\dots$ be the jump times of $X$, which are defined by \cite{norris1998}
\begin{align*}
J(0)
&=0,\\
J(n+1)
&=\inf\{t\ge J(n):X(t)\neq X(J(n))\},\quad n\ge0.
\end{align*}
Further, define the waiting times $S(1),S(2),\dots$ by 
\begin{align*}
S(n)
:=\begin{cases}
J(n)-J(n-1) & \text{if }J(n-1)<\infty,\\
\infty & \text{otherwise},
\end{cases}
\quad n\ge1.
\end{align*}
In words, the $n$th jump of $X$ happens at time $t=J(n)$, and $X$ waits at its new state for time $S(n+1)$.

Assume that the discrete-time process obtained by observing $X$ only at the jump times,
\begin{align}\label{Y}
Y(n)
:=X(J(n)),\quad n\ge0,
\end{align}
is a discrete-time Markov chain. Further, assume that for each $n\ge1$, conditional on $Y(0),\dots,Y(n)$, the waiting times $S(1),\dots,S(n)$ are independent random variables with
\begin{align}\label{both}
\P(S(m)\le t)
=F_{Y(m-1),Y(m)}(t), \quad t\in\R,
\end{align}
where $\{F_{i,j}(t)\}_{i,j\in I}$ are a given set of cumulative distribution functions. In addition, assume that each $F_{i,j}(t)$ satisfies
\begin{align}
F_{i,j}(t)
&=0,\quad t\le {t_{0}(i,j)},\label{aij}\\
F_{i,j}(t)
&\sim {\lambda (i,j)}(t-{t_{0}(i,j)})^{r},\quad\text{as }t\to {t_{0}(i,j)}+,\label{short}
\end{align}
where ${t_{0}(i,j)}\ge0$, ${\lambda (i,j)}>0$, and $r>0$. Assume that
\begin{align*}
\sup_{i,j\in I}{\lambda (i,j)}<\infty,
\end{align*}
which ensures that $X$ cannot take infinitely many jumps in finite time (that is, $X$ is not explosive). 

The assumption in \eqref{both} means that the waiting time $S(n)$ from state $Y(n-1)$ to state $Y(n)$ can depend on both $Y(n-1)$ and $Y(n)$. The assumption in \eqref{aij} means that ${t_{0}(i,j)}\ge0$ is the fastest possible waiting time from state $i$ to state $j$. From a modeling perspective, the benefit of this assumption is that it allows one to ensure that waiting times cannot be arbitrarily small by setting ${t_{0}(i,j)}>0$. The assumption in \eqref{short} describes the waiting time distribution near the fastest waiting time ${t_{0}(i,j)}$.

We note that if $X$ is merely a CTMC as in section~\ref{ctmc}, then 
\begin{align}\label{fij}
\begin{split}
F_{i,j}(t)
&=1-e^{-q(i)t},\quad t\le {t_{0}(i,j)}=0,\\
\lambda(i,j)
&=q(i),\\
r
&=1.
\end{split}
\end{align}
In particular, notice that $F_{i,j}(t)$ in \eqref{fij} depends only on $i$ for the case of a CTMC.

\subsection{Single FPTs}

To describe the short-time distribution of the FPT $\tau$ in \eqref{tau} for this generalized process $X$, we must generalize our definitions in \eqref{path}-\eqref{Lambdathm}. First, let 
\begin{align*}
\Pi=\{\pi(i,j)\}_{i,j\in I}
\end{align*}
be the stochastic matrix governing the discrete-time process $Y(n)$ in \eqref{Y} \cite{norris1998}. In particular, $\pi(i,j)$ is the probability that $Y$ jumps from $i$ to $j$. Define a path $\PP$ of length $d\in\mathbb{Z}_{\ge0}$ from a state $i_{0}\in I$ to a state $i_{d}\in I$ to be a sequence of $d+1$ states in $I$,
\begin{align}\label{path2}
\PP
=(\PP(0),\dots,\PP(d))
=(i_{0},i_{1},\dots,i_{d})\in I^{d+1},
\end{align}
so that
\begin{align}\label{explain2}
\pi(\PP(k),\PP(k+1))
>0,\quad\text{for }k\in\{0,1,\dots,d-1\}.
\end{align}
Note that \eqref{path2}-\eqref{explain2} generalizes the definition of a path in \eqref{path}-\eqref{explain} since if $X$ is a CTMC as in section~\ref{ctmc}, then \cite{norris1998}
\begin{align}\label{piqrela}
\begin{split}
\pi(i,j)
&=\begin{cases}
q(i,j)/q(i) & \text{if }j\neq i\text{ and }q(i)\neq0,\\
0 & \text{if }j\neq i\text{ and }q(i)=0,
\end{cases}\\
\pi(i,i)
&=\begin{cases}
0 & \text{if }q(i)\neq0,\\
1 & \text{if }q(i)=0.
\end{cases}
\end{split}
\end{align}
As in section~\ref{ctmc}, we assume of course that there is a path from the support of $\rho$ to the target (see \eqref{canfind}). We also assume \eqref{away}, which means the searcher cannot start directly on the target.

For a path $\PP\in I^{d+1}$, define $\lambda(\PP)$ to be the product of the $\pi(i,j)$'s and $\lambda(i,j)$'s in \eqref{short} along the path,
\begin{align}\label{lambda2}
\lambda(\PP)
:=\prod_{i=0}^{d-1}\pi(\PP(i),\PP(i+1))\lambda(\PP(i),\PP(i+1))>0.
\end{align}
Note that \eqref{fij} and \eqref{piqrela} imply that \eqref{lambda2} generalizes \eqref{lambda}. In addition, for a path $\PP\in I^{d+1}$, define $t_{0}(\PP)$ to be the sum of the $t_{0}(i,j)$'s in \eqref{short} along the path,
\begin{align*}
t_{0}(\PP)
:=\sum_{i=0}^{d-1}t_{0}(\PP(i),\PP(i+1))
\ge0.
\end{align*}
In words, $t_{0}(\PP)$ is the shortest possible time required to traverse the path $\PP$. Taking the infimum over paths, define the shortest possible time to reach a set $I_{1}\subset I$ starting from a set $I_{0}\subset I$,
\begin{align*}
\amin(I_{0},I_{1})
:=\inf\{t_{0}(\PP):\PP\in I^{d+1},\PP(0)\in I_{0},\PP(d)\in I_{1}\}
\ge0.
\end{align*}
Next, define the smallest number of jumps required to reach $I_{1}$ from $I_{0}$ if the searcher traverses a path $\PP$ with the minimum required time $t_{0}(\PP)=\tmin(I_{0},I_{1})$,
\begin{align*}
\dmin(I_{0},I_{1})
:=\inf\{d:\PP(0)\in I_{0},\PP(d)\in I_{1},t_{0}(\PP)=\tmin(I_{0},I_{1})\}.
\end{align*}
Further, define 
\begin{align*}
\mathcal{S}(I_{0},I_{1})
:=\{\PP\in I^{d+1}:\PP(0)\in I_{0},\PP(d)\in I_{1},t_{0}(\PP)=\tmin(I_{0},I_{1}),d=\dmin(I_{0},I_{1})\}.
\end{align*}
In words, $\mathcal{S}(I_{0},I_{1})$ are the paths $\PP$ going from $I_{0}$ to $I_{1}$ which (i) have the minimum time and (ii) have the minimum number of jumps out of the paths which have the minimum time.
Define
\begin{align*}
\Lambda(\rho,I_{1})
:=\sum_{\PP\in\mathcal{S}(\textup{supp}(\rho),I_{1})}\rho(\PP(0))\lambda(\PP).
\end{align*}
It is immediate that these definitions generalize the definitions in section~\ref{ctmc} since $t_{0}(\PP)=0$ for every path in the case that $X$ is a CTMC.

\begin{proposition}\label{cdf2}
We have that
\begin{align*}
\P(\tau\le {\tmin}+t)
\sim\frac{(\Gamma(r+1))^{d}}{\Gamma(dr+1)}\Lambda t^{d}\quad\text{as }t\to0+,
\end{align*}
where $r>0$ is in \eqref{short} and
\begin{align*}
{\tmin}
&={\amin}(\textup{supp}(\rho),I_{{\t}})\ge0,\\
d
&=\dmin(\textup{supp}(\rho),I_{{\t}})\in\mathbb{Z}_{>0},\\
\Lambda
&=\Lambda(\rho,I_{{\t}})>0.
\end{align*}
\end{proposition}

\subsection{Extreme FPTs}

Having determined the short-time distribution of a single FPT $\tau$ in Proposition~\ref{cdf2}, we now determine the distribution and moments of the fastest FPT, $T_{N}$, and the $k$th fastest FPT, $T_{k,N}$, out of $N\gg1$ iid realizations of $\tau$.

\begin{theorem}\label{main2}
Let $\tmin\ge0$, $d\ge1$, $r>0$, and $\Lambda>0$ be as in Proposition~\ref{cdf2} and define
\begin{align*}
A=\frac{(\Gamma(r+1))^{d}}{\Gamma(dr+1)}\Lambda>0.
\end{align*}
The following rescaling of $T_{N}-{\tmin}$ converges in distribution to a Weibull random variable,
\begin{align*}
(AN)^{1/d}(T_{N}-{\tmin})
\to_{\dist}
\textup{Weibull}(1,d)\quad\text{as }N\to\infty.
\end{align*}
Suppose further that
\begin{align*}
\E[T_{N}]<\infty\quad\text{for some }N\ge1.
\end{align*}
Then for each moment $m\in(0,\infty)$, we have that
\begin{align*}
\E[(T_{N}-{\tmin})^{m}]
&\sim \frac{\Gamma(1+m/d)}{(AN)^{m/d}}\quad\text{as }N\to\infty.
\end{align*}
\end{theorem}

\begin{theorem}\label{kth2}
Fix $k\ge1$ and let $A$ be as in Theorem~\ref{main2}. The following rescaling of $T_{k,N}-\tmin$ converges in distribution to a generalized Gamma random variable,
\begin{align*}
(AN)^{1/d}(T_{k,N}-\tmin)
\to_{\dist}
\textup{gen}\Gamma(1,d,k)\quad\text{as }N\to\infty.
\end{align*}
Suppose further that
\begin{align*}
\E[T_{N}]<\infty\quad\text{for some }N\ge1.
\end{align*}
Then for each moment $m\in(0,\infty)$, we have that
\begin{align*}
\E[(T_{k,N}-\tmin)^{m}]
&\sim\frac{\Gamma(k+m/d)/\Gamma(k)}{(AN)^{m/d}}
\quad\text{as }N\to\infty.
\end{align*}
\end{theorem}

\section{\label{numerics}Numerical simulations}

In this section, we compare the results of our analysis to numerical simulations on complex networks. We consider the setup of section~\ref{ctmc} in which each searcher moves according to a CTMC.

To create the CTMC, we create a graph by randomly connecting $V=|I|\gg1$ vertices by $E$ directed edges (we construct the graph so that $E\approx5V$). We then assign jump rates to each directed edge independently according to a uniform distribution. More precisely, if the CTMC has infinitesimal generator matrix $Q=\{q(i,j)\}_{i,j\in I}$, then the diagonal entries, $q(i,i)\le0$ are chosen so that $Q$ has zero row sums (see \eqref{od2}), and the off-diagonal entries, $q(i,j)\ge0$ with $i\neq j$, are
\begin{align*}
q(i,j)
=\begin{cases}
U_{i,j} & \text{if there is a directed edge from $i$ to $j$},\\
0 & \text{otherwise},
\end{cases}
\end{align*}
where $\{U_{i,j}\}_{i,j\in I}$ are independent uniform random variables on $[0,1]$.

To numerically compute the distribution and mean of the fastest FPT $T_{N}$, we need only compute the survival probability $\P(\tau>t)$ of a single FPT $\tau$ since
\begin{align}
\P(T_{N}>t)
&=\big(\P(\tau>t)\big)^{N},\label{calc1}\\
\E[T_{N}]
&=\int_{0}^{\infty}\big(\P(\tau>t)\big)^{N}\,\dd t.\label{calc2}
\end{align}
To compute $\P(\tau>t)$, let the target be a single node, $I_{\t}=i_{\t}\in I$, and let $\widetilde{Q}$ denote the matrix obtained by deleting the row and column in $Q$ corresponding to $i_{\t}$. Similarly, for an initial distribution $\rho$, let $\widetilde{\rho}$ denote the vector obtained by deleting the entry in $\rho$ corresponding to $i_{\t}$. Then, $\P(\tau>t)$ is given by the sum of the entries in the vector $e^{\widetilde{Q}^{\top}t}\widetilde{\rho}$, where $\widetilde{Q}^{\top}$ denotes the transpose of $\widetilde{Q}$ and $e^{\widetilde{Q}^{\top}t}$ denotes the matrix exponential \cite{lawley2019imp}. In particular, we can write $\P(\tau>t)$ as the dot product,
\begin{align}\label{dot}
\P(\tau>t)
 =\mathbf{1}\cdot e^{\widetilde{Q}^{\top}t}\widetilde{\rho},
\end{align}
where $\mathbf{1}\in\R^{V-1}$ is the vector of all ones.

\begin{figure}[t]
\centering
\includegraphics[width=.6\linewidth]{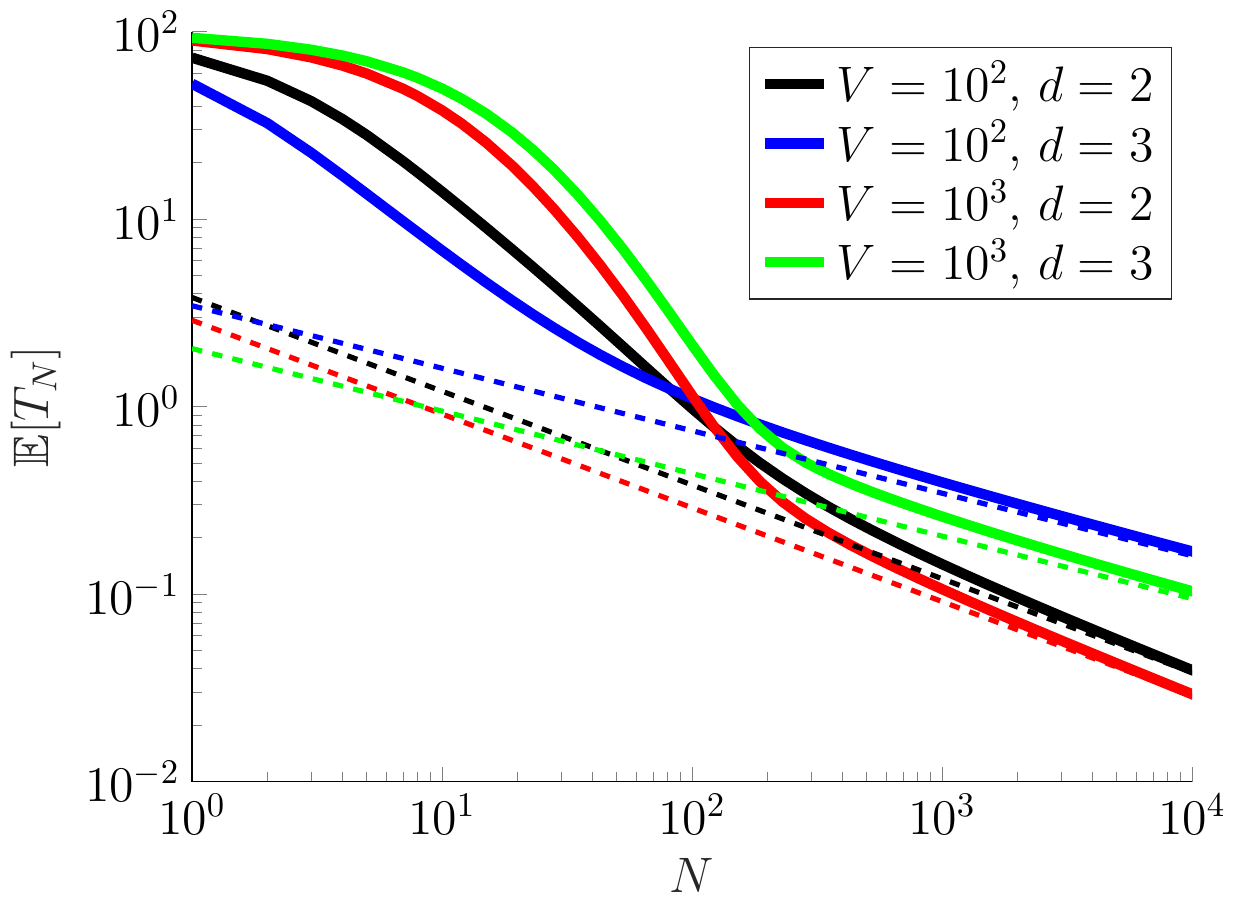}
\caption{\small Mean fastest FPT, $\E[T_{N}]$, as a function of the number of searchers, $N$. The different colored solid curves are $\E[T_{N}]$ computed numerically for different random graphs (different CTMCs). These curves approach the theoretical values (see Theorem~\ref{main} and \eqref{theory}) which are the dashed lines.}
\label{figmfpt}
\end{figure}

In Figure~\ref{figmfpt}, we plot the mean fastest FPT, $\E[T_{N}]$, as a function of the number of searchers, $N$, for different values of the number of vertices $V$ and the shortest distance $d$ from the starting location to the target state. The solid curves are $\E[T_{N}]$ computed from \eqref{calc2}, with $\P(\tau>t)$ computed from \eqref{dot}. The dashed lines are the large $N$ formula for $\E[T_{N}]$ found in Theorem~\ref{main}, namely
\begin{align}\label{theory}
\frac{\Gamma(1+1/d)}{(AN)^{1/d}}.
\end{align}

In agreement with the theory, the solid curves in Figure~\ref{figmfpt} approach the corresponding dashed lines as $N$ increases. In particular, this plot illustrates that the MFPT of a single searcher ($\E[T_{1}]=\E[\tau]\approx10^{2}$) is much slower than the MFPT of the fastest searcher out of many searchers ($\E[T_{N}]\ll \E[\tau]$ if $N\gg10^{2}$).


\begin{figure}[t]
\centering
\includegraphics[width=.6\linewidth]{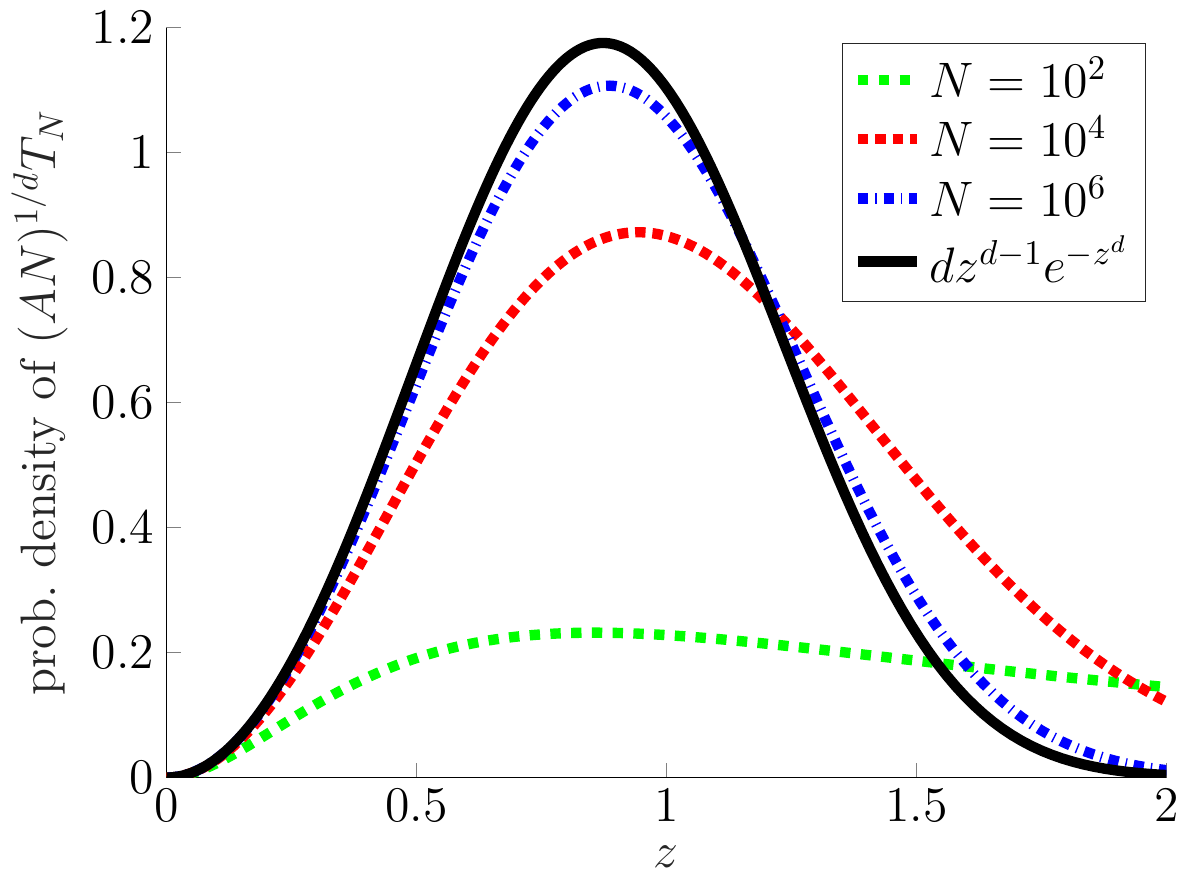}
\caption{\small Probability density of the rescaled fastest FPT, $(AN)^{1/d}T_{N}$, for different values of the number of searchers, $N$. The black solid curve is the probability density of a Weibull random variable with unit scale parameter and shape parameter $d$. In this plot, the random graph has $V=10^{3}$ vertices (states in the CTMC) and the shortest distance from the starting state to the target is $d=3$.}
\label{figdist}
\end{figure}

In addition to the moments of $T_{N}$, Theorem~\ref{main} gives the full probability distribution of $T_{N}$ for large $N$. We illustrate this convergence in Figure~\ref{figdist} by plotting the probability density of the rescaled fastest FPT, $(AN)^{1/d}T_{N}$, for different values of $N$. The probability density of $(AN)^{1/d}T_{N}$ is computed from \eqref{calc1}. In this plot, the graph has $V=10^{3}$ vertices (states for the Markov chain) and the shortest distance from the starting location to the target is $d=3$. In agreement with the theory, the probability density of $(AN)^{1/d}T_{N}$ approaches the density of a Weibull random variable with unit scale parameter and shape parameter $d$ (namely, the limiting density is $dz^{d-1}e^{-z^{d}}$).

\section{\label{diffusion}(Sub)Diffusive searchers}
\subsection{Diffusion}

In this subsection, we compare our results to extreme FPTs for continuous state space diffusion processes, which have been studied extensively \cite{weiss1983, yuste2000, yuste2001, redner2014, meerson2015, ro2017, godec2016x, hartich2018, hartich2019, basnayake2019, lawley2020esp, lawley2020uni, lawley2020dist, lawley2020sub}. Let $X^{\diff}=\{X^{\diff}(t)\}_{t\ge0}$ be a one-dimensional continuous state space diffusion process starting at the origin with diffusivity $D>0$. That is, suppose $X^{\diff}$ satisfies the stochastic differential equation,
\begin{align*}
\dd X^{\diff}
&=\sqrt{2D}\,\dd W,\quad X(0)=0,
\end{align*}
where $W=\{W(t)\}_{t\ge0}$ is a standard Brownian motion. Hence, the probability density $p(x,t)$ that $X(t)=x$ satisfies the Fokker-Planck equation,
\begin{align}
\frac{\partial}{\partial t}p
&=D\frac{\partial^{2}}{\partial x^{2}}p,\quad x\in\R,\,t>0,\label{fpe}\\
p(x,0)
&=\delta(x).\nonumber
\end{align}

Let $\tau^{\diff}$ be the first time that $X^{\diff}$ escapes the interval $(-L,L)$,
\begin{align}\label{td}
\tau^{\diff}
:=\inf\{t>0:X^{\diff}(t)\notin(-L,L)\},
\end{align}
and define the extreme FPT,
\begin{align}\label{tnd}
T_{N}^{\diff}
:=\min\{\tau_{1}^{\diff},\dots,\tau_{N}^{\diff}\},
\end{align}
where $\tau_{1}^{\diff},\dots,\tau_{N}^{\diff}$ are $N$ iid realizations of $\tau^{\diff}$. It is well-known that \cite{weiss1983}
\begin{align}\label{Tdiff}
\E[T_{N}^{\diff}]
\sim\frac{L^{2}}{4D\ln N},\quad\text{as }N\to\infty.
\end{align}
Indeed, the asymptotic behavior in \eqref{Tdiff} holds in much greater generality, including for diffusion processes on $d$-dimensional manifolds with space-dependent diffusivities and force fields \cite{lawley2020uni}.

It is interesting to compare the behavior in \eqref{Tdiff} for diffusive searchers to the behavior we found in Theorem~\ref{main} for CTMCs. Suppose we discretize space with step size
\begin{align}\label{dx}
\Delta x
=\frac{L}{d}
\ll L,
\end{align}
where $d$ is a large natural number. Let $\{X^{\ctmc}(t)\}_{t\ge0}$ be the CTMC which takes values in the one-dimensional network,
\begin{align}\label{1d}
I
:=\{i\Delta x\}_{i\in\mathbb{Z}}
=\{\dots,-\Delta x,0,\Delta x,\dots\},
\end{align}
and has jump rates (using the notation of section~\ref{ctmc}),
\begin{align*}
q(i,i\pm1)
=\frac{q(i)}{2}
&=\frac{D}{(\Delta x)^{2}}
=\frac{D}{L^{2}}d^{2}>0,\quad i\in\mathbb{Z}.
\end{align*}
Assume $X^{\ctmc}(0)=0$. Depending on the context, $X^{\diff}$ may be viewed as an approximation of $X^{\ctmc}$, or vice versa. The correspondence between $X^{\diff}$ and $X^{\ctmc}$ is perhaps most easily seen by noticing that if we use a centered, second order finite difference approximation for the spatial derivative in the Fokker-Planck equation \eqref{fpe} for $X^{\diff}$, then we obtain the master equation (Kolmogorov forward equation) for $X^{\ctmc}$.

Define $\tau^{\ctmc}$ and $T_{N}^{\ctmc}$ analogously to $\tau^{\diff}$ and $T_{N}^{\diff}$,
\begin{align}\label{tndc}
\begin{split}
\tau^{\ctmc}
&:=\inf\{t>0:X^{\ctmc}(t)\notin \{i\Delta x\}_{|i|<d}\},\\
T_{N}^{\ctmc}
&:=\min\{\tau_{1}^{\ctmc},\dots,\tau_{N}^{\ctmc}\},
\end{split}
\end{align}
where $\tau_{1}^{\ctmc},\dots,\tau_{N}^{\ctmc}$ are $N$ iid realization of $\tau^{\ctmc}$. 
Theorem~\ref{main} above implies that
\begin{align}\label{Tctmc}
\E[T_{N}^{\ctmc}]
\sim \frac{L^{2}}{D}\frac{f(d)}{N^{1/d}},\quad\text{as }N\to\infty,
\end{align}
where
\begin{align*}
f(d)
:=\frac{(d!)^{1/d}\Gamma(1+1/d)}{d^{2}}
\sim\frac{1}{de},\quad\text{as }d\to\infty.
\end{align*}
Hence, while the distributions of $X^{\diff}(t)$ and $X^{\ctmc}(t)$ can be made close for any fixed $t\ge0$ by taking $d$ large, we see from \eqref{Tdiff} and \eqref{Tctmc} that the extreme FPTs of $X^{\diff}$ and $X^{\ctmc}$ are quite different for any $d\ge1$.

Put another way, this shows that the diffusion limit, $d\to\infty$, and the many searcher limit, $N\to\infty$, of $T_{N}^{\ctmc}$ do not commute. From a modeling perspective, this means that care must be taken in choosing a model of diffusion (spatially continuous $X^{\diff}$ versus spatially discrete $X^{\ctmc}$) if the system depends on extreme statistics. See \cite{lawley2019pdmp} for an analysis of extreme statistics of diffusion modeled by a piecewise deterministic Markov process (i.e.\ a velocity jump process).

\subsection{Subdiffusion}

In this subsection, we compare our results to extreme FPTs for subdiffusive processes. A subdiffusive process $\{X(t)\}_{t\ge0}$ is defined by a mean-squared displacement that grows sublinearly in time \cite{sokolov2012},
\begin{align*}
\E\Big[\big(X(t)-X(0)\big)^{2}\Big]
\propto t^{\alpha},\quad \alpha\in(0,1).
\end{align*}
A common model for subdiffusion is a certain type of CTRW \cite{metzler2000}. In one space dimension, this model is characterized by a jump length probability density function (pdf), $l(x)$, and a waiting time pdf, $w(t)$. In particular, if the searcher lands at some position $Y(n)\in\R$, the searcher waits until a time chosen from $w(t)$, then jumps to a new location $Y(n+1)=Y(n)+\xi(n+1)$, where $\xi(n+1)$ is chosen from $l(x)$. The searcher continues this process indefinitely.

Assume that the jump length pdf $l(x)$ is symmetric about the origin so that the walk is unbiased, and assume that it has finite standard deviation,
\begin{align*}
\Delta x
:=\sqrt{\int_{-\infty}^{\infty}x^{2}l(x)\,\dd x}<\infty.
\end{align*}
In addition, assume that the waiting time pdf has a slow power-law decay,
\begin{align}\label{tail}
w(t)
\sim C_{\alpha}\Big(\frac{\Delta t}{t}\Big)^{1+\alpha},\quad\text{as }t\to\infty,
\end{align}
where $\alpha\in(0,1)$, for some timescale $\Delta t$ and some rate $C_{\alpha}>0$. Choose $l(x)$ so that $\Delta x=L/d$ as in \eqref{dx} and choose $w(t)$ so that $\Delta t$ satisfies
\begin{align}\label{dt44}
(\Delta t)^{\alpha}
=\frac{(\Delta x)^{2}}{2K_{\alpha}}
=\frac{(L/d)^{2}}{2K_{\alpha}},
\end{align}
where $K_{\alpha}>0$ is some fixed generalized diffusivity. Then, in the diffusion limit $d\to\infty$, it is well-known that the pdf of the limiting process satisfies the fractional Fokker-Planck equation \cite{metzler2000},
\begin{align}\label{fractional}
\frac{\partial}{\partial t}p
=\D K_{\alpha}\frac{\partial^{2}}{\partial x^{2}}p,\quad x\in\R,\,t>0,
\end{align}
where $\D$ is the fractional derivative of Riemann-Liouville type \cite{samko1993}, defined by
\begin{align*}
\D f(t)
=\frac{1}{\Gamma(\alpha)}\frac{\dd}{\dd t}\int_{0}^{t}\frac{f(s)}{(t-s)^{1-\alpha}}\,\dd s.
\end{align*}

Let $X^{\sub}=\{X^{\sub}(t)\}_{t\ge0}$ denote the subdiffusive process starting at the origin,
\begin{align*}
X^{\sub}(0)=0,
\end{align*}
whose pdf satisfies the fractional equation~\eqref{fractional} (note that $X^{\sub}$ can be constructed as a random time change of $X^{\diff}$ \cite{magdziarz2007}). Define $\tau^{\sub}$ and $T_{N}^{\sub}$ analogously to \eqref{td} and \eqref{tnd}. It was recently proven \cite{lawley2020sub} that 
\begin{align}\label{leading}
\E[T_{N}^{\sub}]
\sim\frac{t_{\alpha}}{(\ln N)^{2/\alpha-1}}
\quad\text{as }N\to\infty,
\end{align}
where $t_{\alpha}>0$ is the timescale,
\begin{align*}
t_{\alpha}
:=\Big(\alpha^{\alpha}(2-\alpha)^{2-\alpha}\frac{L^{2}}{4K_{\alpha}}\Big)^{1/\alpha}>0.
\end{align*}

The CTRW leading to the fractional equation~\eqref{fractional} can be put in the framework of section~\ref{general} above. In particular, consider a process $\{X^{\ctrw}(t)\}_{t\ge0}$ with waiting time pdf $w(t)$ satisfying \eqref{tail}-\eqref{dt44} and jump length pdf $l(x)$ given by a sum of Dirac delta functions,
\begin{align*}
l(x)
=\frac{1}{2}\delta(\Delta x-x)+\frac{1}{2}\delta(\Delta x+x),\quad x\in\R,
\end{align*}
with $\Delta x = L/d$. In the notation of section~\ref{general}, the state space $I$ is the discrete set in \eqref{1d}, the jump chain follows
\begin{align*}
\pi(i,j)
=\begin{cases}
1/2 & \text{if $|i-j|=1$},\\
0 & \text{otherwise},
\end{cases}
\end{align*}
and
\begin{align*}
F_{i,i\pm1}(t)
&=\int_{0}^{t}w(t')\,\dd t',\quad i\in\mathbb{Z}.
\end{align*}
Suppose that $\lim_{t\to0+}w(t)=2\lambda(d)>0$ for some function $\lambda(\cdot)$ so that $F(t):=F_{i,i\pm1}(t)\sim2\lambda(d) t$ as $t\to0+$. Thus $F$ satisfies \eqref{aij}-\eqref{short} with $r=1$ and $t_{0}=0$.

Therefore, in the diffusion limit $d\to\infty$, the pdf of the limiting process satisfies \eqref{fractional}, and thus the extreme FPTs satisfy \eqref{leading}. That is, if we take $d\to\infty$ first, and then take $N\to\infty$ limit, then we obtain \eqref{leading}. However, Theorem~\ref{main2} above shows that if we take $N\to\infty$ first for the CTRW $X^{\ctrw}$, then we obtain that the extreme FPTs satisfy
\begin{align}\label{leading2}
\E[T_{N}^{\ctrw}]
\sim\Big[\frac{\Gamma(1+1/d)}{(d!)^{1/d}\lambda(d)}\Big]\frac{1}{N^{1/d}},\quad\text{as }N\to\infty.
\end{align}
where $T_{N}^{\ctrw}$ is defined analogously to \eqref{tndc}. 

Comparing \eqref{leading} and \eqref{leading2}, we again see that the diffusion limit ($d\to\infty$) and the many searcher limit ($N\to\infty$) do not commute. In addition, comparing \eqref{Tctmc} and \eqref{leading2} shows that the extreme FPTs of the discrete state space diffusive process $X^{\ctmc}$ and the discrete state space subdiffusive process $X^{\ctrw}$ both decay as $N^{-1/d}$ as $N\to\infty$. In fact,
\begin{align*}
\E[T_{N}^{\ctmc}]
\sim\E[T_{N}^{\ctrw}],\quad\text{as }N\to\infty,
\end{align*}
if we take $\lambda(d)=(D/L^{2})d^{2}$. Hence, the behavior of extreme statistics is very different in the discrete case ($d<\infty$) compared to the continuum limit ($d=\infty$).

\section{\label{mortal}Fast inactivation of mortal walkers}

Compared to a single FPT, we found in sections~\ref{ctmc} and \ref{general} that extreme FPTs are faster, less variable, and less affected by network size/structure. In essence, considering only the fastest FPTs filters out searchers which deviate from a direct route to the target. It was recently shown in \cite{ma2020, lawley2020mortal} that \emph{fast inactivation} can have a similar effect on FPTs by filtering out slow searchers. These two prior works considered searchers which move by continuous state space diffusion \cite{ma2020, lawley2020mortal} or discrete state space diffusion \cite{ma2020}. In this section, we consider fast inactivation for searchers on networks which move according to a CTMC as in section~\ref{ctmc} or a CTRW as in section~\ref{general}. 

Consider a single searcher that can be inactivated (degrade/die/evanesce/etc.)\ before reaching the target. Such finite lifetime searchers are called ``mortal'' or ``evanescent'' and have been widely studied \cite{abad2010, abad2012, abad2013, yuste2013, meerson2015b, meerson2015, grebenkov2017, ma2020, lawley2020mortal}. Indeed, mortal searchers have been used to model a variety of systems, including inactivation of intracellular signaling molecules \cite{ma2020}, sperm cells searching for an egg despite a high mortality rate \cite{meerson2015}, animals or bacteria foraging for food, extinction of a fluorescent signal in bio-imaging methods, messenger RNA searching for a ribosome, and storage of nuclear waste \cite{grebenkov2017}.

Mathematically, in addition to the FPT $\tau$ of a single searcher (as in \eqref{tau0}), one introduces an independent and exponentially distributed inactivation time $\sigma$ with rate $\gamma>0$,
\begin{align*}
\sigma
=_{\dist}\text{exponential}(\gamma).
\end{align*}
Hence, the event $\tau<\sigma$ means that the searcher found the target before it was inactivated, while $\sigma<\tau$ corresponds to the opposite scenario. Consider the $m$th moment of $\tau$, conditioned that the searcher finds the target before it is inactivated,
\begin{align}\label{conditional}
\E[\tau^{m}\,|\,\tau<\sigma]
:=\frac{\E[\tau^{m}1_{\tau<\sigma}]}{\P(\tau<\sigma)},
\end{align}
where $1_{\tau<\sigma}$ denotes the indicator function,
\begin{align*}
1_{\tau<\sigma}
:=\begin{cases}
1 & \text{if }\tau<\sigma,\\
0 & \text{otherwise}.
\end{cases}
\end{align*}

As in \cite{ma2020, lawley2020mortal}, we are interested in the behavior of the conditional FPT moments \eqref{conditional} in the limit of fast inactivation, i.e.\ $\gamma\to\infty$. The following theorem gives this behavior in terms of the short-time behavior of the unconditioned FPT $\tau$. In particular, Theorem~\ref{easy} is stated for an arbitrary random variable $\tau$ satisfying a certain assumption about its short-time distribution. The subsequent corollaries then consider the case that $\tau$ is a CTMC FPT as in section~\ref{ctmc} (Corollary~\ref{cor1}) and the case that $\tau$ is a CTRW FPT as in section~\ref{general} (Corollary~\ref{cor2}). 

\begin{theorem}\label{easy}
Let $\tau$ be any random variable satisfying
\begin{align*}
\P(\tau\le \tmin+t)
\sim At^{d}\quad\text{as }t\to0+,
\end{align*}
for some $\tmin\ge0$, $A>0$, and $d>0$. Let $\sigma$ be an independent exponential random variable with rate $\gamma>0$ and let $m>0$. If $\tmin=0$, then
\begin{align}\label{zero}
\E[\tau^{m}\,|\,\tau<\sigma]
\sim\frac{\Gamma(d+m)}{\Gamma(d)}\frac{1}{\gamma^{m}}\quad\text{as }\gamma\to\infty.
\end{align}
If $\tmin>0$, then
\begin{align}\label{positive}
\E[\tau^{m}\,|\,\tau<\sigma]-(\tmin)^{m}
\sim\frac{dm}{\gamma}(\tmin )^{m-1}\quad\text{as }\gamma\to\infty.
\end{align}
\end{theorem}

The next two corollaries follow immediately from Theorem~\ref{easy} and Propositions~\ref{cdf} and \ref{cdf2}.

\begin{corollary}\label{cor1}
Let $\tau$ be as in section~\ref{ctmc} and let $d\ge1$ be as in Theorem~\ref{main}. Let $\sigma$ be an independent exponential random variable with rate $\gamma$ and let $m>0$. Equation~\eqref{zero} holds.
\end{corollary}

\begin{corollary}\label{cor2}
Let $\tau$ be as in section~\ref{general} and let $d\ge1$ and $\tmin\ge0$ be as in Theorem~\ref{main2}. Let $\sigma$ be an independent exponential random variable with rate $\gamma$ and let $m>0$. If $\tmin=0$, then \eqref{zero} holds. If $\tmin>0$, then \eqref{positive} holds.
\end{corollary}

Theorem~\ref{easy} and Corollaries~\ref{cor1} and \ref{cor2} show that FPTs conditioned to be less than a fast inactivation time and extreme FPTs have similar qualitative properties. In particular, compared to unconditioned FPTs, such conditional FPTs are faster, less variable (all the moments vanish), and are less affected by the network size/structure/details, since they depend only on the minimum number of jumps $d$ that are required to reach the target. Corollary~\ref{cor1} recovers some results proven in \cite{ma2020} for a discrete state space diffusion model.

\section{\label{discussion}Discussion}

We have analyzed extreme FPTs for a general class of CTRWs on networks. In the case that there are many searchers (random walkers), we found explicit formulas for the extreme FPT distribution and moments that depend only on the parameters along the geodesic path(s) from the starting location(s) to the target. Hence, the extreme FPTs are independent of the details of the network outside this geodesic(s). We proved similar results for searchers which are conditioned to find the target before a fast inactivation time.

Extreme FPTs have been studied extensively for diffusion processes with continuous state spaces \cite{weiss1983, yuste2000, yuste2001, redner2014, meerson2015, ro2017, godec2016x, hartich2018, hartich2019, basnayake2019, lawley2020esp, lawley2020uni, lawley2020dist, lawley2020sub}. This project was started in 1983 by Weiss, Shuler, and Lindenberg \cite{weiss1983}, and more recent work has been motivated primarily by biological applications \cite{schuss2019}. Interesting work has also been done for extreme FPTs of diffusion on fractals \cite{yuste1997, yuste1998}. These previous works are marked by an inverse logarithmic decay of the mean extreme FPT,
\begin{align}\label{uni}
\E[T_{N}]
\propto\frac{1}{\ln N},
\end{align}
as the number of searchers $N$ grows. Indeed, it was recently proven \cite{lawley2020uni} that \eqref{uni} holds for diffusive search under very general assumptions, as long as the searchers cannot start arbitrarily close to the target. If the diffusive searchers start uniformly in the spatial domain (which means that they can start arbitrarily close to the target), then it was proven in \cite{lawley2020comp} that as $N\to\infty$,
\begin{align}\label{other}
\E[T_{N}]
\propto\frac{1}{N^{2}}
\quad\text{or}\quad
\E[T_{N}]
\propto\frac{1}{N}, 
\end{align}
depending on whether the target is perfectly or partially reactive (the result $\E[T_{N}]
\propto N^{-2}$ for a perfectly reactive target was in fact first shown in \cite{weiss1983}).

In contrast to \eqref{uni}-\eqref{other}, in the present work we found that extreme FPTs of CTRWs on discrete networks decay as
\begin{align}\label{new}
\E[T_{N}]
\propto\frac{1}{N^{1/d}},\quad\text{as }N\to\infty,
\end{align}
where $d\ge1$ is the minimum number of jumps required to reach the target. Comparing \eqref{uni} and \eqref{new}, it is clear that the behavior of extreme FPTs of ``diffusion'' depend critically on whether the diffusion is modeled by a continuous state space or a discrete state space. See section~\ref{diffusion} above for more on this discrepancy.

An interesting related work studying extreme FPTs for processes on discrete networks is that of Weng and colleagues \cite{weng2017}. In \cite{weng2017}, the authors investigated the mean of extreme FPTs for \emph{discrete-time} random walks on finite networks (termed the ``mean first parallel passage time''). These authors found an exact formula for this mean time in terms of a matrix describing the network structure. Then, upon averaging over starting locations and target locations, they found that this ``global'' mean time decays as $1/N$ as the number of searchers $N$ grows. 

An important line of related works is \cite{braunstein2003, gautreau2007, gautreau2008, brockmann2013, iannelli2017}, which study various network-based measures which generalize the concept of distance. These measures define the ``effective distance'' between pairs of nodes in a network by taking into account the probabilities of paths between the nodes. Some of this work seeks to understand the arrival time of an infectious disease to a given location. In particular, these works seek to incorporate the idea that frequently traveled routes between two locations (such as airports) make them effectively closer.

In the present work, we similarly found that a certain geodesic path between the source node(s) and the target node(s) controls the extreme FPTs. We found that the geodesic path that is relevant for extreme FPTs minimizes the number of intermediary nodes between the source and target (and the minimum time for the general model in section~\ref{general}). We emphasize that this is a result of the analysis and not an assumption. Indeed, one might have expected that other notions of ``optimal'' paths \cite{braunstein2003} or most probable paths would yield the paths taken by the fastest searchers. However, we have found that this is not the case, as the probability of a path or the rates along a path play a strictly secondary role for extreme FPTs.

Another related work is a novel study of the transport efficiency of the endoplasmic reticulum \cite{dora2020}, which modeled the endoplasmic reticulum as an active network. These authors found a remarkable mode of transportation, in which molecules group together in seemingly redundant packets at particular locations in the network. Similar to the present work, these authors then found that the extreme FPT out of these many apparently redundant molecules is much faster than a single FPT. The extreme FPTs in \cite{dora2020} were computed by making a diffusion approximation and applying results for extreme FPTs of diffusion (which yielded an inverse logarithmic decay of the mean extreme FPT as in \eqref{uni}).

One final related work is the recent study of Ma and colleagues \cite{ma2020}, which considered the effects of a fast inactivation time on mortal diffusive searchers in the context of intracellular signaling. These authors found that if a signaling molecule is conditioned to reach the nucleus before a fast inactivation time, then the FPT is much faster, much less variable, and much less affected by intracellular geometry/obstacles (compared to unconditioned, immortal searchers). As in the case of extreme statistics, such conditioning filters out searchers which deviate from the shortest path to the target. Mathematically, Ref.~\cite{ma2020} considered both continuous state space and discrete state space models of diffusive search, and Ref.~\cite{lawley2020mortal} later considered this problem for continuous state space diffusive search. Corollaries~\ref{cor1} and \ref{cor2} in section~\ref{mortal} above extend some results in \cite{ma2020} to the case of general CTMCs and CTRWs on networks.

\subsection*{Acknowledgments}

The author gratefully acknowledges support from the National Science Foundation (DMS-1944574, DMS-1814832, and DMS-1148230).

\section{\label{appendix}Appendix}

In this appendix, we prove the propositions and theorems of the main text. We begin with a lemma giving the short-time behavior of the cumulative distribution function of a sum of independent random variables. 

\begin{lemma}\label{lemmaexp}
If $\{\sigma_{m}\}_{m=0}^{d-1}$ are independent random variables with
\begin{align*}
\P(\sigma_{m}\le t)
\sim\lambda_{m}t^{r},\quad\text{as }t\to 0+,
\end{align*}
for $\lambda_{m}>0$ and $r>0$, then
\begin{align*}
\P\Big(\sum_{m=0}^{d-1}\sigma_{m}\le t\Big)
\sim \frac{(\Gamma(r+1))^{d}\prod_{m=0}^{d-1}\lambda_{m}}{\Gamma(dr+1)}t^{dr},\; \text{as }t\to0+.
\end{align*}
\end{lemma}

\begin{proof}[Proof of Lemma~\ref{lemmaexp}]
If $X$ is a random variable with cumulative distribution function $F_{X}(t):=\P(X\le t)$, then let $\L^{*}$ denote the Laplace-Stieltjes transform,
\begin{align*}
\L^{*} F
=\{\L^{*} F\}(s)
:=\E[e^{-sX}].
\end{align*}
Letting
\begin{align*}
F_{\sigma_{d}}(t)
&:=\P(\sigma_{d}\le t),\\
F_{\sum_{m=0}^{d-1}\sigma_{d}}(t)
&:=\P\Big(\sum_{m=0}^{d-1}\sigma_{d}\le t\Big),
\end{align*}
independence implies that
\begin{align}\label{indep}
\{\L^{*}F_{\sum_{m=0}^{d-1}\sigma_{d}}\}(s)
=\E[e^{-s\sum_{m=0}^{d-1}\sigma_{d}}]
&=\prod_{m=0}^{d-1}\E[e^{-s\sigma_{d}}]
=\prod_{m=0}^{d-1}\{\L^{*}F_{\sigma_{d}}\}(s).
\end{align}
By the Tauberian theorem (see, for example, Theorems 1 and 3 in chapter XIII.5 in \cite{feller1968}), we have that
\begin{align*}
\{\L^{*}F_{\sigma_{d}}\}(s)
\sim\Gamma(r+1)\lambda_{m}s^{-r},\quad\text{as }s\to\infty.
\end{align*}
Therefore, \eqref{indep} implies that
\begin{align*}
\{\L^{*}F_{\sum_{m=0}^{d-1}\sigma_{d}}\}(s)
\sim(\Gamma(r+1))^{d}s^{-dr}\prod_{m=0}^{d-1}\lambda_{m},\quad\text{as }s\to\infty.
\end{align*}
Applying the Tauberian theorem again yields
\begin{align*}
F_{\sum_{m=0}^{d-1}\sigma_{d}}(t)
\sim \frac{(\Gamma(r+1))^{d}\prod_{m=0}^{d-1}\lambda_{m}}{\Gamma(dr+1)}t^{dr},\quad\text{as }t\to0+,
\end{align*}
which completes the proof.
\end{proof}

\begin{proof}[Proof of Proposition~\ref{cdf}]

We first prove the proposition for the case that $\rho(i_{0})=1$ for some fixed $i_{0}\in I$ and $\rho(j)=0$ for all $j\neq i_{0}$. That is, assume $X(0)=i_{0}$ almost surely.

Let $M(t)\in\mathbb{N}\cup\{0\}$ be the number of jumps of $X(t)$ before time $t$. Then
\begin{align}\label{start}
\P(\tau\le t)
=\P(\tau\le t\cap M(t)\ge d),
\end{align}
since $X$ cannot reach the target from state $i_{0}$ unless it makes at least $d=\dmin(i_{0},I_{\t})\ge1$ jumps. Since $I$ is countable, there are countably many paths of length $d$ from $i_{0}$ to $I_{\t}$. We can thus index the paths so that
\begin{align*}
\mathcal{S}(i_{0},I_{\t})
=\{\PP_{k}\}_{k\in K},
\end{align*}
where $K\subseteq\mathbb{N}$ is some index set. Let $\mathcal{E}_{k}$ denote the event that $X$ takes path $k\in K$ from $i_{0}$ to $I_{\t}$. Notice that
\begin{align}\label{notice}
\P(\mathcal{E}_{k})
=\prod_{i=0}^{d-1}\frac{q(\PP_{k}(i),\PP_{k}(i+1))}{q(\PP_{k}(i))}
=\frac{\lambda(\PP_{k})}{\prod_{i=0}^{d}q(\PP_{k}(i))},\; k\in K.
\end{align}
Now,
\begin{align*}
\P(\tau\le t\cap M(t)\ge d\,|\,\mathcal{E}_{k})
&=\P(M(t)\ge d\,|\,\mathcal{E}_{k})
=\P\Big(\sum_{m=0}^{d-1}\sigma_{m}/q(\PP_{k}(m))<t\Big),
\end{align*}
where $\{\sigma_{m}\}_{m=0}^{d-1}$ are iid exponential random variables with unit rate. Hence, Lemma~\ref{lemmaexp} and \eqref{notice} imply that
\begin{align}\label{from}
\P(\tau\le t\cap M(t)\ge d\,|\,\mathcal{E}_{k})\P(\mathcal{E}_{k})
\sim\frac{\lambda(\PP_{k})}{d!}t^{d},\quad\text{as }t\to0+.
\end{align}

We want to conclude from \eqref{from} that 
\begin{align}
\P(\tau\le t\cap M(t)\ge d)
&=\sum_{k\in K}\P(\tau\le t\cap M(t)\ge d\,|\, \mathcal{E}_{k})\P(\mathcal{E}_{k})\nonumber\\
&\sim\sum_{k\in K}\frac{\lambda(\PP_{k})}{d!}t^{d},\quad\text{as }t\to0+.\label{dct}
\end{align}
If $|K|<\infty$, then this is immediate. To handle the case that $|K|=\infty$, notice that Lemma~\ref{lemmaexp} implies that
\begin{align*}
\P(\tau\le t\cap M(t)\ge d\,|\, \mathcal{E}_{k})
&=\P\Big(\sum_{m=0}^{d-1}\sigma_{m}/q(\PP_{k}(m))<t\Big)\\
&\le\P\Big(\sum_{m=0}^{d-1}\sigma_{m}/\overline{q}<t\Big)
\sim\frac{\overline{q}^{d}}{d!}t^{d},\quad \text{as }t\to0+,
\end{align*}
where $\overline{q}:=\sup_{i}q(i)<\infty$. Therefore, there exists an $\eps>0$ that is independent of $k\in K$ so that
\begin{align}\label{a00}
\begin{split}
&\frac{\P(\tau\le t\cap M(t)\ge d\,|\, \mathcal{E}_{k})\P(\mathcal{E}_{k})}{t^{d}}
\le2\frac{\overline{q}^{d}}{d!}\P(\mathcal{E}_{k}),\quad\text{for all }t\in(0,\eps),\,k\in K.
\end{split}
\end{align}
Since
\begin{align}\label{ab00}
\sum_{k\in K}2\frac{\overline{q}^{d}}{d!}\P(\mathcal{E}_{k})
=2\frac{\overline{q}^{d}}{d!}<\infty,
\end{align}
Lebesgue's dominated convergence theorem yields \eqref{dct}, which then completes the proof for the case $\rho(i_{0})=1$ due to \eqref{start}.

To handle the case of a general initial distribution $\rho$ on $I$, observe that
\begin{align}\label{yep0}
\P(\tau\le t)
=\sum_{i\in \text{supp}(\rho)}\P(\tau\le t\,|\,X(0)=i)\rho(i).
\end{align}
The desired result is then immediate if the support of $\rho$ is finite. A similar application of the dominated convergence theorem as above completes the proof for the case that $I$ is infinite. In particular, as in \eqref{a00}-\eqref{ab00} we have that if $t\in(0,\eps)$, then
\begin{align*}
&t^{-d}\P(\tau\le t\,|\,X(0)=i)\rho(i)
=\rho(i)\sum_{k\in K}t^{-d}\P(\tau\le t\,|\,X(0)=i\cap \mathcal{E}_{k})\P(\mathcal{E}_{k})
\le\rho(i)2\frac{\overline{q}^{d}}{d!}.
\end{align*}
Since $\sum_{i\in\text{supp}(\rho)}\rho(i)2\overline{q}^{d}/d!=2\overline{q}^{d}/d!<\infty$, the proof is complete.
\end{proof}

\begin{proof}[Proof of Theorem~\ref{main}]
The result follows directly from Proposition~\ref{cdf} and Theorems~2 and 3 in \cite{lawley2020comp}.
\end{proof}

\begin{proof}[Proof of Theorem~\ref{kth}]
The result follows directly from Proposition~\ref{cdf} and Theorems~5 and 6 in \cite{lawley2020comp}.
\end{proof}

\begin{proof}[Proof of Proposition~\ref{cdf2}]
As in the proof of Proposition~\ref{cdf}, we first consider the case that $\rho(i_{0})=1$ for some fixed $i_{0}\in I$ and $\rho(j)=0$ for all $j\neq i_{0}$. That is, we assume $X(0)=i_{0}$ almost surely.

Since $I$ is finite, it is immediate that there exists $\eps>0$ so that $t_{0}(\PP)>{\tmin}+\eps$ for all $\PP$ with $t_{0}(\PP)\neq\tmin(i_{0},I_{\t})$. Hence, if $\mathcal{E}$ denotes the event that $X$ takes a path $\PP$ with $t_{0}(\PP)\neq\tmin(i_{0},I_{\t})$, then
\begin{align*}
\P(\tau\le {\tmin}+t\cap\mathcal{E})
=0,\quad\text{for all }t<\eps.
\end{align*}
Therefore, if we index all the paths $\PP$ with $t_{0}(\PP)=\tmin(i_{0},I_{\t})$ as $\{\PP_{k}\}_{k\in K}$ and let $\mathcal{E}_{k}$ denote the event that $X$ takes path $\PP_{k}$, then
\begin{align*}
\P(\tau\le {\tmin}+t)
\sim\sum_{k\in K}\P(\tau\le {\tmin}+t\cap\mathcal{E}_{k}),\;\text{as }t\to0+.
\end{align*}
Let $M(t)\in\mathbb{N}\cup\{0\}$ be the number of jumps of $X(t)$ before time $t$. Now,
\begin{align*}
&\P(\tau\le\tmin+t\cap\mathcal{E}_{k})
=\P(\tau\le\tmin+t\cap M(\tmin+t)\ge d \cap\mathcal{E}_{k}),\quad k\in K,
\end{align*}
since if $X$ takes path $\PP_{k}$ with $k\in K$, then it must make at least $d=\dmin(i_{0},I_{\t})\ge1$ jumps to reach the target.

Next, let $K'\subseteq K$ be an index set so that $\{\PP_{k}\}_{k\in K'}$ are the set of paths in $\mathcal{S}(i_{0},I_{\t})$. It is then immediate that
\begin{align*}
&\P(\tau\le {\tmin}+t\cap M(\tmin+t)\ge d\,|\,\mathcal{E}_{k})
=\P(M(\tmin+t)\ge d\,|\,\mathcal{E}_{k}),
\quad k\in K'.
\end{align*}
Further, if $X$ takes path $\PP_{k}$ with $k\in K'$ and $M(\tmin+t)\ge d$, then
\begin{align*}
\sum_{m=0}^{d-1}\sigma_{m}\le {\tmin}+t,
\end{align*}
where $\{\sigma_{m}\}_{m=0}^{d-1}$ are the $d\ge1$ waiting times. In particular, $\sigma_{m}$ has the distribution
\begin{align*}
\P(\sigma_{m}\le s)
=F_{\PP_{k}(m),\PP_{k}(m+1)}(s).
\end{align*}
Therefore, if we define
\begin{align*}
\widetilde{\sigma}_{m}
=\sigma_{m}-t_{0}(\PP_{k}(m),\PP_{k}(m+1)),
\end{align*}
then $\P(\widetilde{\sigma}_{m}\le s)=\widetilde{F}_{\PP_{k}(m),\PP_{k}(m+1)}(s)$, where
\begin{align*}
&\widetilde{F}_{\PP_{k}(m),\PP_{k}(m+1)}(s)
:=F_{\PP_{k}(m),\PP_{k}(m+1)}(s+t_{0}(\PP_{k}(m),\PP_{k}(m+1))),
\end{align*}
and thus as $s\to0+$,
\begin{align*}
\widetilde{F}_{\PP_{k}(m),\PP_{k}(m+1)}(s)
\sim\lambda(\PP_{k}(m),\PP_{k}(m+1))s^{r}. 
\end{align*}
Therefore, by Lemma~\ref{lemmaexp}, we have that for $k\in K'$,
\begin{align*}
&\P(\tau\le {\tmin}+t\cap M(\tmin+t)\ge d\,|\,\mathcal{E}_{k})\\
&=\P\Big(\sum_{m=0}^{d-1}\sigma_{m}\le {\tmin}+t\Big)\\
&=\P\Big(\sum_{m=0}^{d-1}\big[t_{0}(\PP_{k}(m),\PP_{k}(m+1))+\widetilde{\sigma}_{m}\big]\le {\tmin}+t\Big)\\
&=\P\Big(\sum_{m=0}^{d-1}\widetilde{\sigma}_{m}\le t\Big)
\sim
\frac{(\Gamma(r+1))^{d}\lambda(\PP_{k})}{\Gamma(dr+1)}t^{dr},\; \text{as }t\to0+.
\end{align*}
Noting that $\P(\mathcal{E}_{k})=\pi(\PP_{k})$ and summing over $k$ completes the proof for the case $\rho(i_{0})=1$ (note that $|K|<\infty$ since $|I|<\infty$). The case of a general distribution $\rho$ on $I$ is handled analogously to \eqref{yep0}.
\end{proof}

\begin{proof}[Proof of Theorem~\ref{main2}]
The result follows directly from Proposition~\ref{cdf2} and Theorems~2 and 3 in \cite{lawley2020comp}.
\end{proof}

\begin{proof}[Proof of Theorem~\ref{kth2}]
The result follows directly from Proposition~\ref{cdf2} and Theorems~5 and 6 in \cite{lawley2020comp}.
\end{proof}

\begin{proof}[Proof of Theorem~\ref{easy}]
Lemma~3 in \cite{lawley2020mortal} gives the following representation for the conditional $m$th moment,
\begin{align}\label{mortallemma}
\begin{split}
\E[\tau^{m}\,|\,\tau<\sigma]
&=\frac{\int_{0}^{\infty} t^{m} \gamma e^{-\gamma t}F(t)\,\dd t}{\int_{0}^{\infty}\gamma e^{-\gamma t}F(t) \,\dd t}
-\frac{\int_{0}^{\infty}mt^{m-1} e^{-\gamma t}F(t)\,\dd t}{\int_{0}^{\infty}\gamma e^{-\gamma t}F(t)\,\dd t},
\end{split}
\end{align}
where $F(t):=\P(\tau\le t)$. First, suppose that $\tmin=0$. Let $\eps\in(0,1)$. By assumption, there exists $\delta>0$ so that
\begin{align*}
(1-\eps)At^{d}
<F(t)
<(1+\eps)At^{d},\quad\text{for all }t\in(0,\delta).
\end{align*}
Therefore, for any $n>-1$, we have that
\begin{align}\label{en}
\begin{split}
(1-\eps)\int_{0}^{\delta}t^{n+d}e^{-\gamma t}\,\dd t
&<\frac{1}{A}\int_{0}^{\delta}t^{n}e^{-\gamma t}F(t)\,\dd t
<(1+\eps)\int_{0}^{\delta}t^{n+d}e^{-\gamma t}\,\dd t.
\end{split}
\end{align}
Now, it is a straightforward to check that
\begin{align*}
\int_{0}^{\delta}t^{n+d}e^{-\gamma t}\,\dd t
\sim\int_{0}^{\infty}t^{n+d}e^{-\gamma t}\,\dd t,\quad\text{as }\gamma\to\infty,
\end{align*}
since $\int_{\delta}^{\infty}t^{n+d}e^{-\gamma t}\,\dd t$ vanishes exponentially as $\gamma\to\infty$. Furthermore, it is a simple calculus exercise to check that
\begin{align*}
\int_{0}^{\infty}t^{n+d}e^{-\gamma t}\,\dd t
=\frac{\Gamma(d+n+1)}{\gamma^{d+n+1}}.
\end{align*}
Since $\int_{\delta}^{\infty}t^{n}e^{-\gamma t}F(t)\,\dd t$ vanishes exponentially as $\gamma\to\infty$ and since $\eps>0$ is arbitrary in \eqref{en}, we thus obtain
\begin{align}\label{this}
\int_{0}^{\infty}t^{n}e^{-\gamma t}F(t)\,\dd t
\sim \frac{A\Gamma(d+n+1)}{\gamma^{d+n+1}},\quad\text{as }\gamma\to\infty.
\end{align}
Combining \eqref{this} with \eqref{mortallemma} and simplifying yields \eqref{zero}.

Next, suppose $\tmin>0$. As above, it is straightforward to check that if $n>-1$, then as $\gamma\to\infty$ we have
\begin{align}\label{yep}
\int_{0}^{\infty}t^{n}e^{-\gamma t}F(t)\,\dd t
\sim \int_{\tmin}^{\infty}t^{n}e^{-\gamma t}A(t-\tmin)^{d}\,\dd t.
\end{align}
Furthermore, changing variables yields
\begin{align}\label{yep2}
\int_{\tmin}^{\infty}e^{-\gamma t}(t-\tmin)^{d}\,\dd t
=e^{-\gamma\tmin}\frac{\Gamma(d+1)}{\gamma^{d+1}}.
\end{align}
In addition, a straightforward application of Watson's lemma gives
\begin{align}\label{yep3}
\begin{split}
&\Big(\int_{\tmin}^{\infty}t^{m}e^{-\gamma t}(t-\tmin)^{d}\,\dd t
-e^{-\gamma\tmin}\frac{\Gamma(d+1)(\tmin)^{m}}{\gamma^{d+1}}\Big)\\
&\sim e^{-\gamma\tmin}\frac{m\Gamma(d+2)(\tmin)^{m-1}}{\gamma^{d+2}},\quad\text{as }\gamma\to\infty.
\end{split}
\end{align}
Similarly, Watson's lemma also gives that as $\gamma\to\infty$,
\begin{align}\label{yep4}
\int_{\tmin}^{\infty}t^{m-1}e^{-\gamma t}(t-\tmin)^{d}\,\dd t
\sim e^{-\gamma\tmin}\frac{\Gamma(d+1)(\tmin)^{m-1}}{\gamma^{d+1}}.
\end{align}
Combining \eqref{yep}-\eqref{yep4} with \eqref{mortallemma} completes the proof.
\end{proof}


\bibliography{library.bib}

\begin{thebibliography}{10}

\bibitem{albert2002}
R{\'e}ka Albert and Albert-L{\'a}szl{\'o} Barab{\'a}si.
\newblock Statistical mechanics of complex networks.
\newblock {\em Reviews of modern physics}, 74(1):47, 2002.

\bibitem{newman2003}
Mark~EJ Newman.
\newblock The structure and function of complex networks.
\newblock {\em SIAM review}, 45(2):167--256, 2003.

\bibitem{pastor2015}
Romualdo Pastor-Satorras, Claudio Castellano, Piet Van~Mieghem, and Alessandro
  Vespignani.
\newblock Epidemic processes in complex networks.
\newblock {\em Reviews of modern physics}, 87(3):925, 2015.

\bibitem{noh2004}
Jae~Dong Noh and Heiko Rieger.
\newblock Random walks on complex networks.
\newblock {\em Physical review letters}, 92(11):118701, 2004.

\bibitem{masuda2017}
Naoki Masuda, Mason~A Porter, and Renaud Lambiotte.
\newblock Random walks and diffusion on networks.
\newblock {\em Physics reports}, 716:1--58, 2017.

\bibitem{iannelli2017}
Flavio Iannelli, Andreas Koher, Dirk Brockmann, Philipp H{\"o}vel, and Igor~M
  Sokolov.
\newblock Effective distances for epidemics spreading on complex networks.
\newblock {\em Physical Review E}, 95(1):012313, 2017.

\bibitem{dora2020}
M~Dora and D~Holcman.
\newblock Active flow network generates molecular transport by packets: case of
  the endoplasmic reticulum.
\newblock {\em Proceedings of the Royal Society B}, 287(1930):20200493, 2020.

\bibitem{redner2001}
Sidney Redner.
\newblock {\em A guide to first-passage processes}.
\newblock Cambridge University Press, 2001.

\bibitem{condamin2007b}
S~Condamin, O~B{\'e}nichou, V~Tejedor, R~Voituriez, and Joseph Klafter.
\newblock First-passage times in complex scale-invariant media.
\newblock {\em Nature}, 450(7166):77--80, 2007.

\bibitem{reuveni2010}
Shlomi Reuveni, Rony Granek, and Joseph Klafter.
\newblock Vibrational shortcut to the mean-first-passage-time problem.
\newblock {\em Physical Review E}, 81(4):040103, 2010.

\bibitem{schuss2019}
Z.~Schuss, K.~Basnayake, and D.~Holcman.
\newblock Redundancy principle and the role of extreme statistics in molecular
  and cellular biology.
\newblock {\em Physics of Life Reviews}, January 2019.

\bibitem{coombs2019}
D~Coombs.
\newblock First among equals: Comment on ``{R}edundancy principle and the role
  of extreme statistics in molecular and cellular biology'' by {Z. Schuss, K.
  Basnayake and D. Holcman}.
\newblock {\em Physics of life reviews}, 28:92--93, 2019.

\bibitem{redner2019}
S~Redner and B~Meerson.
\newblock Redundancy, extreme statistics and geometrical optics of brownian
  motion. comment on ``{R}edundancy principle and the role of extreme
  statistics in molecular and cellular biology'' by {Z. Schuss et al.}
\newblock {\em Physics of life reviews}, 28:80--82, 2019.

\bibitem{sokolov2019}
I~M Sokolov.
\newblock Extreme fluctuation dominance in biology: On the usefulness of
  wastefulness: Comment on ``{R}edundancy principle and the role of extreme
  statistics in molecular and cellular biology'' by {Z. Schuss, K. Basnayake
  and D. Holcman.}
\newblock {\em Physics of life reviews}, 2019.

\bibitem{rusakov2019}
D~A Rusakov and L~P Savtchenko.
\newblock Extreme statistics may govern avalanche-type biological reactions:
  Comment on ``{R}edundancy principle and the role of extreme statistics in
  molecular and cellular biology'' by {Z. Schuss, K. Basnayake, D. Holcman}.
\newblock {\em Physics of life reviews}, 2019.

\bibitem{martyushev2019}
L~M Martyushev.
\newblock Minimal time, weibull distribution and maximum entropy production
  principle. comment on ``{R}edundancy principle and the role of extreme
  statistics in molecular and cellular biology'' by {Z. Schuss et al.}
\newblock {\em Physics of life reviews}, 28:83--84, 2019.

\bibitem{tamm2019}
M~V Tamm.
\newblock Importance of extreme value statistics in biophysical contexts:
  Comment on ``{R}edundancy principle and the role of extreme statistics in
  molecular and cellular biology.''.
\newblock {\em Physics of life reviews}, 2019.

\bibitem{basnayake2019c}
Kanishka Basnayake and David Holcman.
\newblock Fastest among equals: a novel paradigm in biology. reply to comments:
  Redundancy principle and the role of extreme statistics in molecular and
  cellular biology.
\newblock {\em Physics of life reviews}, 28:96--99, 2019.

\bibitem{colesbook}
S~Coles.
\newblock {\em An introduction to statistical modeling of extreme values},
  volume 208.
\newblock Springer, 2001.

\bibitem{falkbook}
M~Falk, J~H{\"u}sler, and RD~Reiss.
\newblock {\em Laws of small numbers: extremes and rare events}.
\newblock Springer Science \& Business Media, 2010.

\bibitem{haanbook}
L~De~Haan and A~Ferreira.
\newblock {\em Extreme value theory: an introduction}.
\newblock Springer Science \& Business Media, 2007.

\bibitem{lawley2020uni}
S~D Lawley.
\newblock Universal formula for extreme first passage statistics of diffusion.
\newblock {\em Phys Rev E}, 101(1):012413, 2020.

\bibitem{weiss1983}
G~H Weiss, K~E Shuler, and K~Lindenberg.
\newblock Order statistics for first passage times in diffusion processes.
\newblock {\em J Stat Phys}, 31(2):255--278, 1983.

\bibitem{yuste2000}
SB~Yuste and L~Acedo.
\newblock Diffusion of a set of random walkers in euclidean media. first
  passage times.
\newblock {\em J Phys A}, 33(3):507, 2000.

\bibitem{yuste2001}
S~B Yuste, L~Acedo, and K~Lindenberg.
\newblock Order statistics for $d$-dimensional diffusion processes.
\newblock {\em Phys Rev E}, 64(5):052102, 2001.

\bibitem{redner2014}
S~Redner and B~Meerson.
\newblock First invader dynamics in diffusion-controlled absorption.
\newblock {\em J Stat Mech}, 2014(6):P06019, 2014.

\bibitem{meerson2015}
B~Meerson and S~Redner.
\newblock Mortality, redundancy, and diversity in stochastic search.
\newblock {\em Phys Rev Lett}, 114(19):198101, 2015.

\bibitem{ro2017}
S~Ro and Y~W Kim.
\newblock Parallel random target searches in a confined space.
\newblock {\em Phys Rev E}, 96(1):012143, 2017.

\bibitem{godec2016x}
A~Godec and R~Metzler.
\newblock Universal proximity effect in target search kinetics in the
  few-encounter limit.
\newblock {\em Phys Rev X}, 6(4):041037, 2016.

\bibitem{hartich2018}
D~Hartich and A~Godec.
\newblock Duality between relaxation and first passage in reversible markov
  dynamics: rugged energy landscapes disentangled.
\newblock {\em New J Phys}, 20(11):112002, 2018.

\bibitem{hartich2019}
D~Hartich and A~Godec.
\newblock Extreme value statistics of ergodic markov processes from first
  passage times in the large deviation limit.
\newblock {\em J Phys A}, 52(24):244001, 2019.

\bibitem{basnayake2019}
K~Basnayake, Z~Schuss, and D~Holcman.
\newblock Asymptotic formulas for extreme statistics of escape times in 1, 2
  and 3-dimensions.
\newblock {\em J Nonlinear Sci}, 29(2):461--499, 2019.

\bibitem{lawley2020esp}
S~D Lawley and J~B Madrid.
\newblock A probabilistic approach to extreme statistics of {B}rownian escape
  times in dimensions 1, 2, and 3.
\newblock {\em Journal of Nonlinear Science}, pages 1--21, 2020.

\bibitem{lawley2020dist}
S~D Lawley.
\newblock Distribution of extreme first passage times of diffusion.
\newblock {\em Journal of Mathematical Biology}, 2020.

\bibitem{lawley2020sub}
Sean~D Lawley.
\newblock Extreme statistics of anomalous subdiffusion following a fractional
  fokker-planck equation: Subdiffusion is faster than normal diffusion.
\newblock {\em Journal of Physics A: Mathematical and Theoretical}, 2020.

\bibitem{abad2010}
E~Abad, SB~Yuste, and Katja Lindenberg.
\newblock Reaction-subdiffusion and reaction-superdiffusion equations for
  evanescent particles performing continuous-time random walks.
\newblock {\em Physical Review E}, 81(3):031115, 2010.

\bibitem{abad2012}
E~Abad, SB~Yuste, and Katja Lindenberg.
\newblock Survival probability of an immobile target in a sea of evanescent
  diffusive or subdiffusive traps: A fractional equation approach.
\newblock {\em Phys Rev E}, 86(6):061120, 2012.

\bibitem{abad2013}
E~Abad, SB~Yuste, and Katja Lindenberg.
\newblock Evanescent continuous-time random walks.
\newblock {\em Phys Rev E}, 88(6):062110, 2013.

\bibitem{yuste2013}
SB~Yuste, E~Abad, and Katja Lindenberg.
\newblock Exploration and trapping of mortal random walkers.
\newblock {\em Phys Rev Lett}, 110(22):220603, 2013.

\bibitem{meerson2015b}
Baruch Meerson.
\newblock The number statistics and optimal history of non-equilibrium steady
  states of mortal diffusing particles.
\newblock {\em J Stat Mech: Theory Exp}, 2015(5):P05004, 2015.

\bibitem{grebenkov2017}
D~S Grebenkov and G~Oshanin.
\newblock Diffusive escape through a narrow opening: new insights into a
  classic problem.
\newblock {\em Phys Chem Chem Phys}, 19(4):2723--2739, 2017.

\bibitem{ma2020}
Jingwei Ma, Myan Do, Mark~A Le~Gros, Charles~S Peskin, Carolyn~A Larabell,
  Yoichiro Mori, and Samuel~A Isaacson.
\newblock Strong intracellular signal inactivation produces sharper and more
  robust signaling from cell membrane to nucleus.
\newblock {\em bioRxiv}, 2020.

\bibitem{lawley2020mortal}
Sean~D Lawley.
\newblock The effects of fast inactivation on conditional first passage times
  of mortal diffusive searchers.
\newblock {\em arXiv preprint arXiv:2003.05515}, 2020.

\bibitem{norris1998}
J.R. Norris.
\newblock {\em {Markov Chains}}.
\newblock Statistical {\&} Probabilistic Mathematics. Cambridge University
  Press, 1998.

\bibitem{lawley2019imp}
S~D Lawley and J~B Madrid.
\newblock First passage time distribution of multiple impatient particles with
  reversible binding.
\newblock {\em J Chem Phys}, 150(21):214113, 2019.

\bibitem{lawley2019pdmp}
Sean~D Lawley.
\newblock Extreme first passage times of piecewise deterministic markov
  processes.
\newblock {\em arXiv preprint arXiv:1912.03438}, 2019.

\bibitem{sokolov2012}
Igor~M Sokolov.
\newblock Models of anomalous diffusion in crowded environments.
\newblock {\em Soft Matter}, 8(35):9043--9052, 2012.

\bibitem{metzler2000}
Ralf Metzler and Joseph Klafter.
\newblock The random walk's guide to anomalous diffusion: a fractional dynamics
  approach.
\newblock {\em Physics reports}, 339(1):1--77, 2000.

\bibitem{samko1993}
Stefan~G Samko, Anatoly~A Kilbas, Oleg~I Marichev, et~al.
\newblock {\em Fractional integrals and derivatives}, volume~1.
\newblock Gordon and Breach Science Publishers, Yverdon Yverdon-les-Bains,
  Switzerland, 1993.

\bibitem{magdziarz2007}
Marcin Magdziarz, Aleksander Weron, and Karina Weron.
\newblock Fractional fokker-planck dynamics: Stochastic representation and
  computer simulation.
\newblock {\em Physical Review E}, 75(1):016708, 2007.

\bibitem{yuste1997}
S~Bravo Yuste.
\newblock Escape times of $j$ random walkers from a fractal labyrinth.
\newblock {\em Physical review letters}, 79(19):3565, 1997.

\bibitem{yuste1998}
S~Bravo Yuste.
\newblock Order statistics of diffusion on fractals.
\newblock {\em Physical Review E}, 57(6):6327, 1998.

\bibitem{lawley2020comp}
Jacob~B Madrid and Sean~D Lawley.
\newblock Competition between slow and fast regimes for extreme first passage
  times of diffusion.
\newblock {\em Journal of Physics A: Mathematical and Theoretical}, 2020.

\bibitem{weng2017}
Tongfeng Weng, Jie Zhang, Michael Small, and Pan Hui.
\newblock Multiple random walks on complex networks: A harmonic law predicts
  search time.
\newblock {\em Physical Review E}, 95(5):052103, 2017.

\bibitem{braunstein2003}
Lidia~A Braunstein, Sergey~V Buldyrev, Reuven Cohen, Shlomo Havlin, and
  H~Eugene Stanley.
\newblock Optimal paths in disordered complex networks.
\newblock {\em Physical review letters}, 91(16):168701, 2003.

\bibitem{gautreau2007}
Aur{\'e}lien Gautreau, Alain Barrat, and Marc Barth{\'e}lemy.
\newblock Arrival time statistics in global disease spread.
\newblock {\em Journal of Statistical Mechanics: Theory and Experiment},
  2007(09):L09001, 2007.

\bibitem{gautreau2008}
Aur{\'e}lien Gautreau, Alain Barrat, and Marc Barthelemy.
\newblock Global disease spread: statistics and estimation of arrival times.
\newblock {\em Journal of theoretical biology}, 251(3):509--522, 2008.

\bibitem{brockmann2013}
Dirk Brockmann and Dirk Helbing.
\newblock The hidden geometry of complex, network-driven contagion phenomena.
\newblock {\em science}, 342(6164):1337--1342, 2013.

\bibitem{feller1968}
William Feller.
\newblock {\em An introduction to probability theory and its applications:
  {V}olume {I}}.
\newblock John Wiley \& Sons New York, 3 edition, 1968.

\end{thebibliography}
\bibliographystyle{unsrt}

\end{document}